\documentclass[11pt,letterpaper]{article}

\usepackage[margin=1in]{geometry}
\usepackage{graphicx}
\usepackage{subcaption}
\usepackage{todonotes}
\usepackage[normalem]{ulem}

\usepackage{amsthm}
\usepackage{amsmath}
\usepackage{amsfonts}
\usepackage{amssymb}
\usepackage{dsfont,color,xfrac,enumitem,mathrsfs}
\usepackage{mathtools,tikz-cd}

\usepackage[numbers]{natbib}
\usepackage[colorlinks,citecolor=blue,urlcolor=blue]{hyperref}

\numberwithin{equation}{section}

\newtheorem{theorem}{Theorem}[section]
\newtheorem{lemma}[theorem]{Lemma}
\newtheorem{thm}{Theorem}[section]

\newtheorem{corollary}[theorem]{Corollary}

\theoremstyle{remark}
\newtheorem{definition}[theorem]{Definition}

\newtheorem{rem}[thm]{Remark}
\newtheorem{obs}[thm]{Observation}

\newtheorem{remark}{Remark}

\renewcommand{\le}{\leqslant} 
\renewcommand{\ge}{\geqslant} 
\renewcommand{\leq}{\leqslant} 
\renewcommand{\geq}{\geqslant}

\newcommand{\ind}{\mathds{1}}
 
\newcommand{\eps}{\varepsilon}

\newcommand{\norm}[1]{\left\Vert#1\right\Vert}
\newcommand{\abs}[1]{\left\vert#1\right\vert}

\newcommand{\ie}{\emph{i.e.,}}

\newcommand{\eg}{\emph{e.g.,}}
 
\newcommand{\equald}{\stackrel{\mathrm{d}}{=}}

\def\qed{ \hfill $\blacksquare$}

\let\ga=\alpha    
     \let\gl=\lambda         \let\gs=\sigma

\newcommand{\cA}{\mathcal{A}}
\newcommand{\cF}{\mathcal{F}}
\newcommand{\cG}{\mathcal{G}}

\newcommand{\mvG}{\boldsymbol{G}}

\newcommand{\mvv}{\boldsymbol{v}}
\newcommand{\mvx}{\boldsymbol{x}}

\newcommand{\mveps}{\boldsymbol{\eps}}

\newcommand{\dN}{\mathds{N}}

\newcommand{\dR}{\mathds{R}}

\newcommand{\dZ}{\mathds{Z}} 

\DeclareMathOperator{\E}{\mathds{E}}
\DeclareMathOperator{\pr}{\mathds{P}}

\DeclareMathOperator{\var}{Var}
\DeclareMathOperator{\cov}{Cov}

\DeclareMathOperator{\geom}{Geom}

\newcommand{\ORA}{\overrightarrow}
\newcommand{\OLA}{\overleftarrow}

\newcommand{\olm}{\overline{m}}
\newcommand{\bXi}{\mathbf{\Xi}}
\newcommand{\OSLA}[1]{\stackrel{\mathrel{\rotatebox[origin=c]{180}{$\rightsquigarrow$}}}{#1}}

\title{Scaling Limit of a Stochastic Clustering Model on $\dR$}
\author{Partha S.~Dey \and S.~Rasoul Etesami \and Aditya S.~Gopalan}
\date{\today}

\begin{document}
\maketitle

\begin{abstract}
    We consider an infinite-dimensional stochastic clustering model on $\dR$. In discrete time, each point of a unit-intensity simple point process moves halfway toward either of its left or right neighbors, chosen uniformly at random. Co-located points are merged into a single point, and the resulting simple point process is rescaled to unit intensity. We show that, when the point processes are shifted so that there is a point at the origin, the dynamics have a unique weak limit when the initial point process is renewal.
    For this limiting point process, the gap distribution has exponential tails.
    We also show that for the time-reversed process and with an appropriate scaling in space, there is a limiting (random) distribution function on $\dR$, whose associated measure assigns to $\dR$ a measure corresponding to the gap between consecutive points. Finally, we discuss several relevant research directions.
\end{abstract}

\setcounter{tocdepth}{1}
\tableofcontents

\section{Introduction}
Dynamic clustering refers to a class of stochastic or deterministic dynamics in which a configuration of data points (or particles) evolves over time through local interactions~\cite{aldous1999deterministic}. These interactions cause points to aggregate, merge, split, or reorganize, so that the resulting cluster structure is itself time-dependent and emerges from the underlying dynamics. Dynamic clustering is often done via numerical optimization~\cite{li2021survey}.
For large, finite datasets, there are two challenges with this approach.
First, large numerical optimization problems are computationally expensive.
Second, and perhaps more important, is the question of determining stopping criteria for the clustering algorithm.
For finite datasets, continued use of a clustering algorithm often results in all data points being assigned to a single cluster, which is undesirable.

In this work, we address this second question.
Specifically, we study \emph{stochastic} dynamic clustering algorithms on infinite datasets, aiming to determine whether these dynamics possess a stationary measure. If a stationary measure exists, a natural stopping criterion for the clustering dynamics is when one is close to that stationary measure.
For large finite datasets, one can also accept the current clustering when the cluster gap distribution is close to that of the stationary measure for a corresponding infinite dataset.

To our knowledge, this is the first work to directly analyze this type of clustering dynamics on an infinite dataset, with the intent of informing behavior on large finite datasets. Indeed, we demonstrate this behavior via simulation for two simple algorithms below. The algorithms are as follows.
\begin{itemize}
 \item \textbf{Algorithm 1:} Each point moves halfway towards either its left or right neighbor, chosen uniformly at random. Co-located points are merged. The space is always rescaled so that the point process has unit intensity.
 \item \textbf{Algorithm 2:} Each point moves halfway towards either its left or right neighbor, chosen such that conditionally, its movement is mean-zero. Here also, co-located points are merged. The space is always rescaled so that the point process has unit intensity.
\end{itemize}
Simulation results for the two algorithms are presented in Figure~\ref{fig:SRW}. Surprisingly, these two algorithms appear to exhibit different behavior: for Algorithm 1, the limit distribution of the gaps appears independent of the initial distribution, whereas for Algorithm 2, it appears to depend on the initial distribution of the gaps.

\begin{figure}[htbp]
\centering
 \begin{subfigure}[t]{.45\columnwidth}
  \centering
  \includegraphics[width=\linewidth,trim={.75cm .75cm .75cm .75cm},clip]{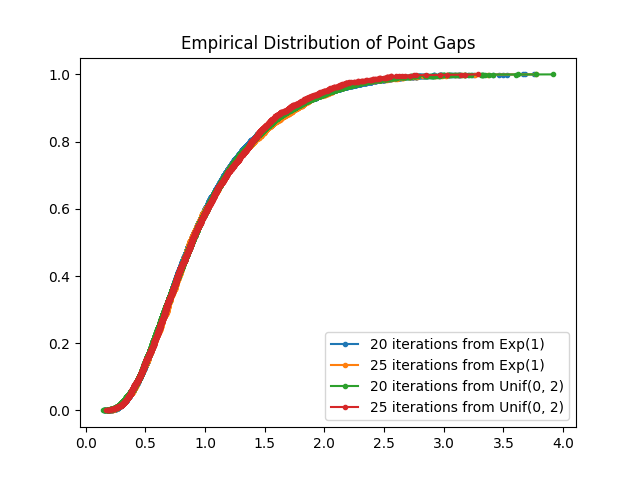}
  \caption{Algorithm 1}
  \label{subfig:srw}
 \end{subfigure}
 \begin{subfigure}[t]{.45\columnwidth}
  \centering
  \includegraphics[width=\linewidth,trim={.75cm .75cm .75cm .75cm},clip]{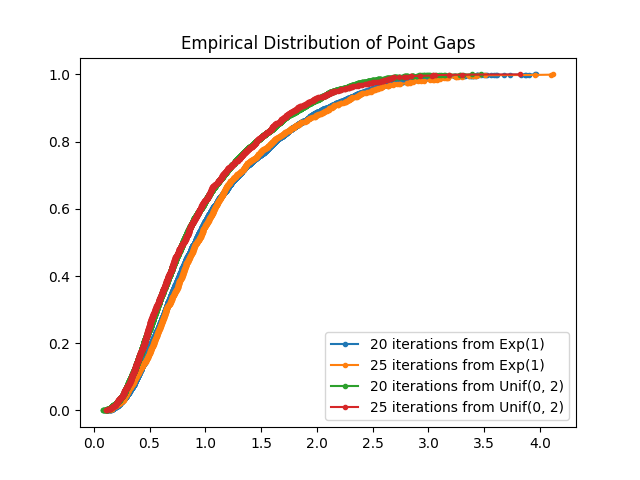}
  \caption{Algorithm 2}
  \label{subfig:zm}
 \end{subfigure}
 \caption{Empirical distribution of point gaps after 20 and 25 iterations, starting from gaps distributed as $\mathrm{Exp}(1)$ and $\mathrm{Unif}(0,2)$.}
 \label{fig:SRW}
\end{figure}

In this paper, we prove that Algorithm 1 has a scaling limit.
That is, (under mild assumptions on the initial data), repeated application of Algorithm 1 converges to a unique limit that is independent of the initial data. Moreover, the size of a typical cluster (the number of points merged into a single point) after appropriate scaling converges to a nontrivial random variable. Finally, we prove the existence of a random measure whose mass gives the limiting scaled gap distribution, and whose length of support gives the limiting scaled cluster size distribution.

Note that while Algorithms 1 and 2 seem similar, even identifying the scaling required in Algorithm 2 is a nontrivial task.
For Algorithm 1, the factor $\frac34$ appears due to the fact that neighboring points will co-locate with probability $\frac14$, independently of their spatial positions.
It is far from obvious what the correct scaling factor should be for Algorithm 2, and even if that scaling factor should be independent of the initial data.
Thus, there is a significant amount of interesting and important future work to be done in terms of understanding the limiting properties of these kinds of dynamics.

\subsection{Related Work}

Clustering problems on static data have been widely studied in various areas, such as the minimum distortion problem in data compression~\cite{gersho2012vector}, facility location problems~\cite{drezner1996facility}, and $k$-means clustering in pattern recognition~\cite{likas2003global}, among many others~\cite{xu2005survey}. However, recent advances in geographic information systems have led to growing interest in developing \emph{dynamic} clustering algorithms for ensembles of moving objects or for data points that evolve over time. Applications of dynamic clustering include coverage control for mobile sensing networks~\cite{cortes2004coverage,sharma2011entropy,xu2013clustering}, the development of automatic deployment and tracking/monitoring algorithms for surveillance systems~\cite{frazzoli2004decentralized,zhang2008continuous}, modeling database demand variability and data placement in web applications~\cite{ghanbari2011tracking,etesami2024distributed}, clustering spatio-temporal brain signal dynamics~\cite{nunez2006electric}, and traffic routing through traffic-flow clustering~\cite{sheu2002fuzzy}.

From a different perspective, dynamic clustering has also emerged in a wide range of multi-agent averaging dynamics, which have been studied extensively from a control-theoretic viewpoint, with applications ranging from distributed coordination and formation~\cite{hendrickx2012convergence,jadbabaie2003coordination}, to modeling animal flocking patterns~\cite{chazelle2014convergence}, and to opinion dynamics~\cite{degroot1974reaching,friedkin1999social,hegselmann2002opinion}. While averaging dynamics were initially used to model consensus in multi-agent systems (e.g., robotic rendezvous~\cite{bullo2009distributed} or consensus in social networks~\cite{degroot1974reaching}), these models were later extended to study disagreement (clustering) in addition to consensus.

For instance, in the context of opinion dynamics, Friedkin and Johnsen were the first to adopt this approach, proposing a dynamical system that captures opinion polarization into multiple clusters~\cite{friedkin1999social}. This model was later extended by Hegselmann and Krause~\cite{hegselmann2002opinion} and belongs to the broader class of bounded-confidence dynamics~\cite{bernardo2024bounded}. Beyond bounded-confidence models, dynamic clustering has also been studied in broader classes of nonlinear and signed interaction systems, where antagonistic or repulsive interactions naturally lead to the emergence of multiple opinion groups~\cite{altafini2013consensus,bernardo2024bounded}. Analyzing the convergence and limit points of such dynamics, however, poses major challenges compared to standard averaging models, whose convergence properties can often be studied using classical tools from Markov chains or graph theory. These difficulties become even more pronounced when additional stochasticity is introduced into the dynamics~\cite{wang2017noisy,dey20252r}.

As another example, averaging dynamics on state-dependent graphs (e.g., where interactions depend on agents' positions) are closely related to dynamic clustering, especially when the data points are spatially distributed and evolve over time, and the goal is to cluster them based on similarity~\cite{hegselmann2002opinion,aggarwal2003framework,Bhatti2019}. Similarly, in dynamic community detection, one aims to understand how diffusion-based averaging dynamics on time-varying graphs can reveal communities (clusters) that persist over time~\cite{delvenne2010stability}. In both cases, the co-evolution of agents' states and network structure introduces dynamic coupling among agents, requiring new analytical approaches to characterize limit points and long-term behavior. These complementary perspectives highlight that dynamic clustering is a robust and structurally rich outcome of state-dependent interactions, extending well beyond classical consensus frameworks.

There are a few references on averaging processes on graphs~\cite{al12}, in which an edge is selected at random, and the incident vertex values are averaged. This model has been generalized to hypergraphs in~\cite{spiro22}. Recent work by~\cite{cdsz22} studies the cutoff phenomenon for the average process on the complete graph, in connection with a question by Bourgain related to quantum computing. However, our model differs from the average process, as all particles move simultaneously at each time step. In Section~\ref{sec:toy}, we give an example of a toy model of follow-the-leader that exhibits behavior similar to the average process and provides motivation for analyzing the dual weight process and its distribution function.

Motivated by these applications, we propose and study a new class of averaging dynamics for spatial data clustering, which can be viewed in the same spirit as the Hegselmann-Krause opinion dynamics~\cite{hegselmann2002opinion}, but over infinitely many agents and under a simplified neighborhood structure. To our knowledge, no work in the data clustering literature directly considers an infinite data set.
There are, however, some recent works in the probability literature that address problems similar to ours.
Baccelli and Khaniha~\cite{khaniha2025hierarchical} consider an infinite-dimensional clustering model on $\dR^2$; however, they are unable to establish a scaling limit for their dynamics.
Angel, Ray, and Spinka~\cite{angel2023tale} study a model on $\dR$ that behaves similarly to clustering, although it is not strictly a clustering model.
They leave the question of whether a scaling limit exists for their dynamics as an open problem. Even identifying the correct scaling is a challenging problem for either of these two papers.
Further, identifying the scaling for Algorithm 2 is also a challenging problem; our results are facilitated by the straightforward scaling for Algorithm 1.

\subsection{Contributions and Organization}

The main contributions of this work are as follows.

\begin{itemize}
 \item We propose and analyze a new class of dynamical clustering algorithms for spatial $\mathbb{R}$-valued data, where each point randomly clusters with its left or right neighbor. To our knowledge, this is the first study to examine such dynamics on an infinite dataset.
 \item For this model, we show that if the initial point process is a renewal process and the point processes are shifted so that a point lies at the origin, a unique weak limit exists for both scaled positions and cluster sizes, which is independent of the initial data.
 \item Our proof combines a time-reversal construction with stochastic duality and a recursive analysis of reverse-time weights, which may be useful for other infinite-dimensional clustering dynamics. The results about concentration for stationary renewal processes may also be of independent interest.
\end{itemize}

Our analysis of Algorithm $1$ relies crucially on the following two properties.
First, we have order preservation: that is, if a point $u$ starts to the left of a point $v$, it will never be to the right of $v$ (this property also holds for Algorithm 2).
Second, the time-reversal dynamics can be described as the composition of certain independent (random) linear operators.
This provides a well-behaved structure from which we can show our results.
Independence in the second property fails for Algorithm 2.

This paper is organized as follows. In Section~\ref{sec:prelim}, we provide relevant background materials and notations.
In Section~\ref{sec:Model}, we describe the model and main results.
We prove the main results in Section~\ref{sec:conv-profs}.
In Section~\ref{sec:tech}, we prove the technical results in this article.
Finally, in Section~\ref{sec:future-work}, we discuss future work.


\section{Background Material and Notations}\label{sec:prelim}

In this section, we first provide definitions and preliminary results, along with the notations used throughout the paper.

\begin{definition}
    A \emph{point process} on $\dR$ is a (possibly random) countable collection of points on $\dR$, which is, without loss of generality, a non-decreasing element of $\dR^\dZ$.
    A point process $\Xi$ is \emph{simple} if $\pr(\text{for all $x \in \dR$, } \text{more than one point of $\Xi$ is at $x$}) = 0$.
    A simple point process $\Xi$ on $\dR$ is \emph{renewal} if the gaps between successive points are independent and identically distributed.
\end{definition}

\begin{definition}
    A point process $\Xi$ on $\dR$ is \emph{non-simple} if $$\pr(\text{for some $x \in \dR$, } \text{more than one point of $\Xi$ is at $x$}) > 0.$$
    For a non-simple point process $\Xi$, the \emph{simplified point process} corresponding to $\Xi$ is a simple point process with exactly one point at each location where $\Xi$ has a point.
\end{definition}

\begin{definition}
    A point process $\Xi$ is \emph{stationary} if for any two sets $S_1, S_2 \subset \dR$ of equal and finite Lebesgue measure, the distributions of the number of points in $S_1$ and $S_2$ are equal. If $\Xi$ is a stationary point process, the \emph{Palm-shifted point process} $\Theta \Xi$ is the point process constructed by shifting all points of $\Xi$ such that the $0$-indexed point of $\Xi$ is at the origin.
\end{definition}
These definitions can be readily restricted to define point processes on $\mathbb{Z}$, and we will do so without further mention.

\begin{definition}
    Let $N(t)$ denote the number of points of a stationary point process $\Xi$ in $(0, t]$, for $0 \neq t \in \dR$. The \emph{intensity} of $\Xi$ is $\gl  := \lim_{t \to \pm\infty}\E N(t)/t$, provided it exists.
\end{definition}
The intensity of a point process can be readily extended to be spatially inhomogeneous or even random, but we will not require that level of generality in this paper; thus, we provide only this restricted definition.

\begin{lemma}[{\cite[Theorem 6.1]{k73}}]
    For Palm-shifted point processes, convergence of finite-dimensional distributions implies tightness.
    Thus, convergence of finite-dimensional distributions is equivalent to weak convergence.
\end{lemma}

\begin{definition}
    Let $\Gamma^{(t)}$ and $\eta^{(t)}$ be Markov processes on measurable spaces $S_1$ and $S_2$, respectively.
    The processes $\Gamma^{(\cdot)}$ and $\eta^{(\cdot)}$ are \emph{stochastically dual} with respect to a function $h: S_1 \times S_2 \to \dR$ if for all $t \in \dZ, \Gamma^{(0)} \in S_1, \eta^{(0)} \in S_2$, we have:
    $$\E_{\Gamma^{(0)}}h(\Gamma^{(t)},\eta^{(0)}) = \E_{\eta^{(0)}}h( \Gamma^{(0)}, \eta^{(t)}).$$
\end{definition}
The intuitive idea of stochastic duality is that if certain functionals of a Markov process are difficult to compute, those same functionals may be more readily computed for a stochastic dual process. In this paper, we use time reversal as the stochastic dual.

\subsection{Notations}
We adopt the following notation in this paper. For a process $X$, $\ORA{X}^{(t)}$ denotes the forward-time evolution, and $\OLA{X}^{(t)}$ denotes the reverse-time evolution. For a point process $\Xi$, $\Theta \Xi$ denotes the Palm-shifted point process. Convergence of random variables will be denoted as follows: $X_n \xrightarrow{\mathrm{f.d.d.}} X$ denotes convergence of finite-dimensional distributions; $X_n \xrightarrow{L^p} X$ denotes convergence in $L^p$; and $X_n \Rightarrow X$ denotes weak convergence.

\section{Model and Main Results}
\label{sec:Model}
\subsection{Point Process Model}
We consider the following model. Let $\bXi$ be the space of unit-intensity stationary point processes on $\dR$. Note that $\Xi \in \bXi$ need not have a point at the origin.

Let $\ORA{\Xi}^{(0)} \in \bXi$.
We first define a point process $\hat{\Xi}^{(t+1)}$ in terms of $\ORA{\Xi}^{(t)}$ as follows.
Each point moves halfway towards one of its neighbors, chosen uniformly at random and independently of other points' movements.
If two points move to the same location, they merge so that $\hat{\Xi}^{(t+1)}$ is a simple stationary point process with intensity $\frac{3}{4}$.
Then, $$
 \ORA{\Xi}^{(t+1)} =
 \frac34\times \hat{\Xi}^{(t+1)}
$$ is obtained by re-scaling $\hat{\Xi}^{(t+1)}$ to unit intensity.
Here, we denote the re-scaling to unit intensity by left-multiplying a point process by its intensity: that is, if $\Xi = (\Xi_i)_{i \in \dZ}$, then $c\times\Xi:= (c\Xi_i)_{i \in \dZ}$, for any constant $c > 0$. We refer to Figure~\ref{fig:tree} for an illustration of our dynamics without rescaling to unit intensity.

We remark that we could instead define the dynamics on the space of non-simple point processes.
In that case, our dynamics preserve the unit-intensity trivially.
However, the resultant non-simple point processes consist of ever-increasing numbers of co-located points at ever-increasing distances, and it is difficult to make sense of a stationary measure.
It is more convenient to scale the simplified point process and endow each point with an $\dN$-valued mark equal to the number of initial points co-located at that point in the simplified point process.

\begin{figure}[htbp]
 \centering
 \includegraphics[width=1\linewidth]{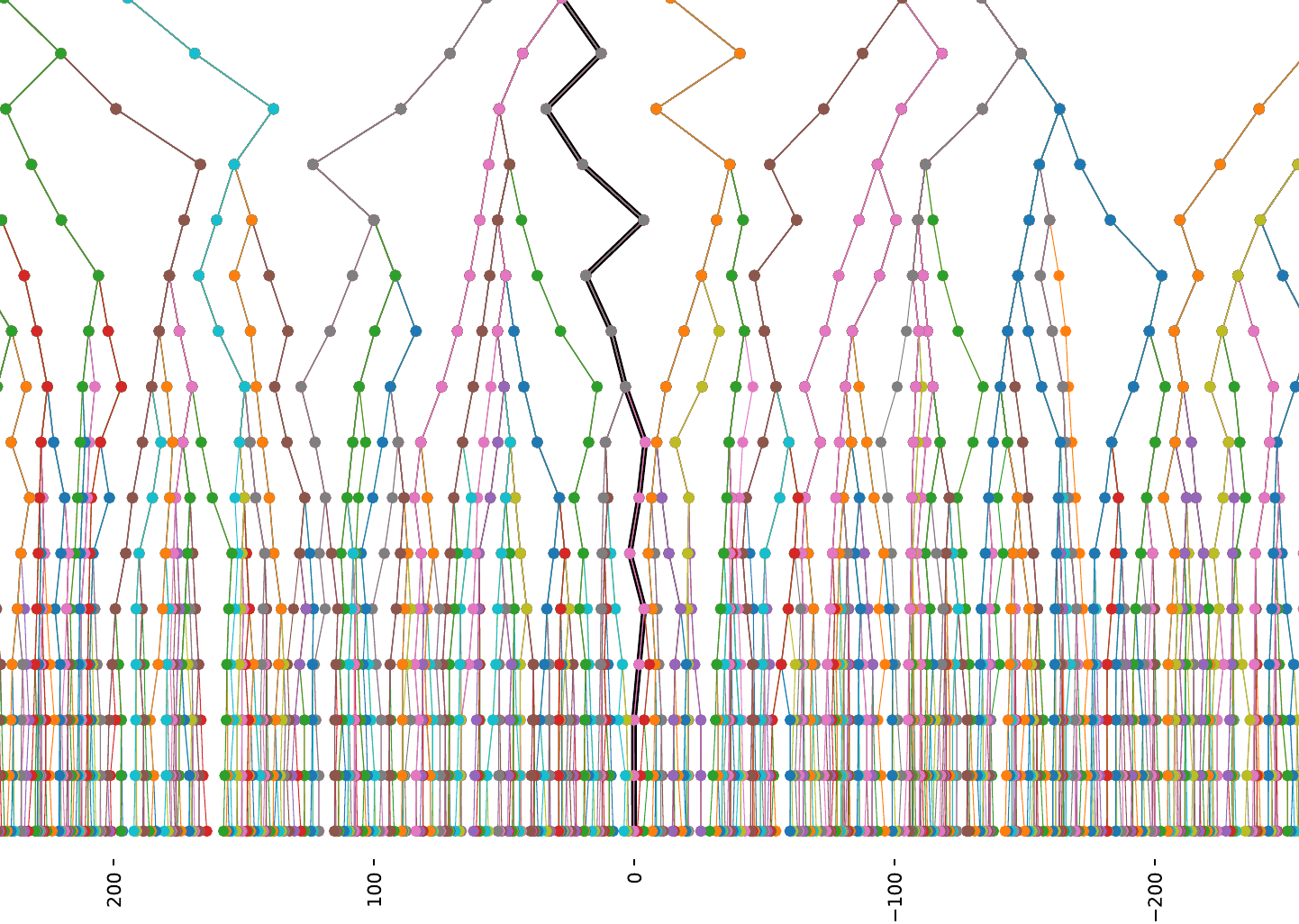}
 \caption{A simulation of the point process model, without space re-scaling.}
 \label{fig:tree}
\end{figure}

\subsubsection{Informal Statements of Results}

Here, we state our main results informally, followed by their formal statements in Section~\ref{ssec:main-results}.

First, we prove that the (Palm-shifted) dynamics of $\Xi^{(t)}$ are ergodic; that is, they admit a unique limit point.
We then show that the tail of the limiting gap distribution is exponential.
We also establish a scaling limit for the number of initial points that have merged with the initially $0$-indexed point.
These results are obtained via a time-reversal argument; in addition, we show that the time-reversed process is a stochastic dual of the original dynamics. 

Our result is about the ergodicity of a certain infinite-dimensional Markov process that describes Algorithm 1.
For dynamic clustering, this result means that independently of the initial data, the distribution near which the clustering should be accepted is the same.
However, the limiting (non-simple before cluster size scaling) point process is \emph{not} renewal: there are dependencies between the size and locations of neighboring clusters that arise.
Thus, the behavior of Algorithm 1 is a ``smoothing'' of the initial data in some sense, but nevertheless it is this ``smoothing'' that allows us to prove the result -- for infinite dimensional Markov processes, it is not typical for a stationary measure to be unique.
Notice that in Figure~\ref{fig:SRW}, this behavior does not seem to hold for Algorithm 2.
This is a more ``natural'' behavior, but it requires the development of fundamentally different tools than the ones we employ here for the simpler Algorithm 1.

In Theorem~\ref{thm:weight-sum-convergence}, we examine the dynamics of the time-reversed process in greater detail to show how the limiting gap distribution can be iteratively constructed; this requires the gap sequence model and its time reversal as defined in Section~\ref{ssec:time-reversal}.

\begin{remark}
    In our model, each point moves towards its right neighbor independently with probability $\frac12$, so the scaling factor is $\frac34 = 1 - \frac12\cdot\frac12$.
    For a model where points move to the right with probability $p$ and to the left with probability $1-p$, one can study it similarly with the scaling $1-p(1-p)$; however, there will be drift at speed $2p-1$.
    This drift is obviously to the right when $p > \frac12$ and to the left when $p < \frac12$, and that can be interpreted from the sign of $2p-1$.
    All of our results have analogues in this more general situation.

    We expect that analogues of our results also hold when each particle moves halfway towards one of its $k$ nearest neighbors in each direction.
    In this setting, we no longer have order conservation: a point $u$ that starts to the left of another point $v$ could end up to the right of $v$.
    This complicates the analysis significantly.
    In the model we study, a point $u$ that starts at another point $v$'s left either remains to the left of $ v$ or co-locates with $v$.
\end{remark}

\subsection{Gap Sequence Model}
We consider a second formulation of the model, in which the dynamics can be considered as a product of certain $\dZ$--indexed random matrices.
In this formulation of the model, rather than tracking the point locations, we track the gaps between points.
The original point process model can be recovered by also tracking a real-valued quantity representing the location of the $0$--indexed point.

Indeed, let $\ORA{\Gamma}^{(t)} \in \dR^\dZ$ be the \emph{gap sequence} associated to $\ORA{\Xi}^{(t)}$; that is,
$$
 \ORA{\Gamma}^{(t)}_i := \ORA{\Xi}^{(t)}_{i+1} - \ORA{\Xi}^{(t)}_i.
$$
Notice that the components of $\ORA{\Gamma}^{(t)}$ are identically distributed, but generally \emph{not} independent. In what follows, we treat $\ORA{\Gamma}^{(t)}$ as a $\dZ$--indexed column vector.

We define two random linear operators, $\cA^{(t)}$ and $\cF^{(t)}$, which we refer to as \emph{averaging} and \emph{folding}, respectively.
One iteration of the dynamics is expressed by the dynamics $$\ORA{\Gamma}^{(t+1)} = \cF^{(t)}\cA^{(t)} \ORA{\Gamma}^{(t)},$$ where averaging corresponds to the point movements, and folding corresponds to merging of co-located points.

To define these operators formally as $\mathbb{Z}^2$-indexed random matrices, let us use the notation $\cA^{(t)}_i, \cF^{(t)}_i$ to denote the $i$-th rows of $\cA^{(t)}, \cF^{(t)}$ respectively, where $i \in \dZ$. Note that we have,
\begin{align}\label{eq:lincomb}
 \ORA{\Gamma}^{(t)} = \cF^{(t-1)}\cA^{(t-1)}\cdots \cF^{(2)}\cA^{(2)} \cF^{(1)}\cA^{(1)} \ORA{\Gamma}^{(0)} \ \ \ \forall t\geq 1,
\end{align}
which implies that for any $i\in\dZ$, the element $\ORA{\Gamma}^{(t)}_i$ can be written as a non-negative random linear combination of the entries in $\ORA{\Gamma}^{(0)}$. Also, let $e_i$ denote the $\mathbb{Z}$-indexed standard basis vector that is equal to $1$ in the $i$-th component and $0$ in all other components.\medskip

\noindent\textbf{Averaging:}
The averaging operator is defined row-wise as follows, which is clear from the dynamics. For all $i\in\dZ$,
\begin{equation*}
 \cA^{(t)}_i :=
 \begin{cases}
  \frac{1}{2}(e_i + e_{i+1}) & \text{ with probability } 1/2, \\
  \frac{1}{2}(e_i + e_{i-1}) & \text{ with probability } 1/2.
 \end{cases}
\end{equation*}
Here, the randomness in each row is independent.\medskip

\noindent\textbf{Folding:}
Observe that point merges at time $t$ occur according to a stationary $\dZ$--valued renewal process $\tau^{(t)}$, and that for $t \neq s$, $\tau^{(t)}$ and $\tau^{(s)}$ are i.i.d. It is easy to see that the inter-renewal distribution of $\tau^{(t)}$ is the sum of two independent $\geom(1/2)$ random variables. From here, it follows that the rows of the folding matrix $\cF^{(t)}$ can be expressed as follows.
For a process $\tau\subseteq\dZ$, let $N_{\tau}(\cdot)$ denote the counting function of $\tau$, where
\begin{align*}
    N_{\tau}(i) = \abs{\tau\cap [0,i]} - \abs{\tau\cap [-i,0)}
 \text{ for } i\in\dZ.
\end{align*}
We have, with $\tau=\tau^{(t)}$,
\begin{equation*}
 \cF^{(t)}_i :=
 \begin{cases}
  \frac34(e_{i+N_\tau(i)-1} + e_{i+N_\tau(i)}) & \text{ for } i\in \tau^{(t)}\cap[0,\infty), \\
  \frac34(e_{i+N_\tau(i)} + e_{i+N_\tau(i)+1}) & \text{ for } i\in \tau^{(t)}\cap(-\infty,0), \\
  \frac34e_{i+N_{\tau}(i)}                     & \text{ for } i \not\in \tau^{(t)}.
 \end{cases}
\end{equation*}
At the indices of the mergings, we need to correspondingly add the two merged gaps; otherwise, we keep the original gap.
To do this, we also need to keep track of the number of mergings across space.

\subsection{Time Reversal of Gap Sequence Model}
\label{ssec:time-reversal}
The forward-time gap sequence model consists of averaging and then folding.
Thus, the time reversal consists of ``un--folding,'' and then ``un--averaging.'' The reverse-time process is as follows.

Recall that in the forward dynamics, $1/4$ of the points are lost due to merging. In the time reversal, $1/3$ of the points should thus split into two new points.
Indeed, in the forward process, the interval between points that merge with their left neighbors is given by the sum of two independent $\geom(1/2)$ random variables. Thus, in the time-reversed process, points ``un-merge'' with interval $\geom(1/2) + \geom(1/2) - 1$, where the geometrics are again independent. Thus, let $\rho^{(t)}$ be i.i.d.~renewal processes with gap distribution
$$
    Y \sim \geom(1/2) + \geom(1/2) - 1,
$$
where the geometric random variables are independent and take the minimum value $1$. We have $\E Y = 2+2-1=3$. The requisite coupling between the merged indices and the point processes $\rho^{(t)}$ used in the un-merging is obvious, and we omit it for the sake of brevity; see Figure~\ref{fig:FR-dynamics} instead.
\begin{figure}[htbp]
 \centering
 \includegraphics[width=0.8\linewidth]{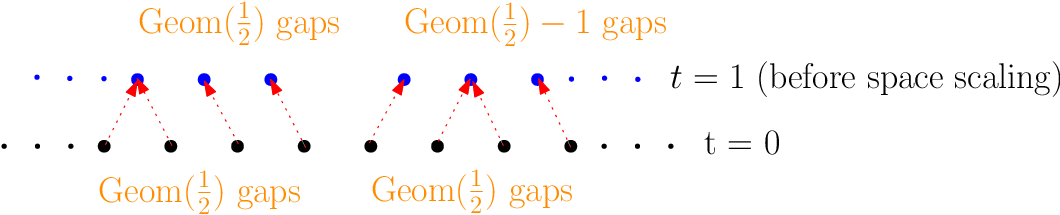}
 \caption{Figure of forward and reverse dynamics.}
 \label{fig:FR-dynamics}
\end{figure}

Let $\cG$ denote the space of nonnegative integer-valued sequences indexed by $\dZ$. The reverse-time dynamics can be expressed in terms of a $\cG$--valued Markov process $\eta^{(t)}$.
The interpretation of $\eta^{(t)}$ is that each element $\eta^{(t)}_i$ is a \emph{weight}, and the gaps at time $t$ can be constructed by summing and re-scaling the appropriate subset of weights.
The exposition is cleaner with integer-valued weights; since averaging induces a factor of $\frac{1}{2}$, we will need to re-scale the integer-valued weights by $\frac{3}{8}$, instead of $\frac{3}{4}$, to compensate for this.
The process $\OLA{\eta}^{(t)} \in \dZ_+^\dZ$ evolves as follows.
\begin{itemize}
 \item  For indices $i \not\in \rho^{(t+1)}$, replace the element $\OLA{\eta}^{(t)}_i$ by the tuple $(\OLA{\eta}^{(t)}_i, \OLA{\eta}^{(t)}_i)$. This corresponds to ``un-averaging.''
 \item  For indices $i \in \rho^{(t+1)}$, replace the element $\OLA{\eta}^{(t)}_i$ by the tuple $(\OLA{\eta}^{(t)}_i, 2\OLA{\eta}^{(t)}_i,\OLA{\eta}^{(t)}_i)$. This corresponds to ``un-folding'' and then ``un-averaging.''
 \item  Finally, construct a new sequence $\OLA{\eta}^{(t+1)}$ by adding the right-most element of the $i$-th tuple with the left-most element of the $(i+1)$-th tuple and removing the parentheses.
\end{itemize}

Here, this summation corresponds to the fact that the same points contribute to two different gaps after one iteration of the time-reversed dynamics.
The sequence $\OLA{\eta}^{(t+1)}$ is indexed so that the $0$--indexed point is the sum of the left-most element of the $0$--th tuple and the right-most element of the $(-1)$--st tuple.
For example, in Figure~\ref{fig:FR-dynamics}, the un-mergings correspond to the entire gap that gets deleted in the merging, but also half of each of the preceding and following gaps.

\subsection{Main Results}
\label{ssec:main-results}

\begin{theorem}\label{thm:gap-sequence-convergence}
    Let $\ORA{\Xi}^{(0)}$ be a renewal process with finite inter-renewal variance.
    The following holds:
 \begin{equation}
  \ORA{\Xi}^{(t)} \xrightarrow{\text{f.d.d.}} \ORA{\Xi}^{(\infty)},
  \label{eq:HK-convergence}\end{equation}
    where the limit is independent of $\ORA{\Xi}^{(0)}$.
    If $\Theta$ is the Palm shift, then
 \begin{align*}
  \Theta\ORA{\Xi}^{(t)} \Rightarrow \Theta\ORA{\Xi}^{(\infty)}.
 \end{align*}
    Moreover, the gap distribution of $\ORA{\Xi}^{(\infty)}$ has an exponentially decaying tail.
\end{theorem}

We do not know the distribution of the Palm shift. The limit is not renewal, and thus it is difficult to identify; we leave it as an open problem. We use an appropriate martingale structure for the time-reversed process to establish Theorem~\ref{thm:gap-sequence-convergence}. Using the standard Burkholder-Davis-Gundy inequality for discrete-time martingales gives an $O(p^{3/2})$ upper bound for the $p$-th moment, which is not enough to get exponential decay.  To get an exponential tail decay, we directly control the MGF. The following corollary is a direct consequence of the proof technique of Theorem~\ref{thm:gap-sequence-convergence}.

\begin{corollary} \label{cor:duality}
    The time-reversed process $\left({3}/{8}\right)^t\OLA{\eta}^{(t)}$ is the stochastic dual of the gap sequence model, with respect to the inner product
    $$
        h(\ORA{\Gamma}^{(s)}, \left({3}/{8}\right)^t\OLA{\eta}^{(t)}) := \left\langle \ORA{\Gamma}^{(s)}, \left({3}/{8}\right)^t\OLA{\eta}^{(t)} \right\rangle = \sum_{i \in \dZ}\left(\frac{3}{8}\right)^t\ORA{\Gamma}^{(s)}_i\OLA{\eta}^{(t)}_i.
    $$
\end{corollary}

\begin{theorem} \label{thm:cluster-size-convergence}
    Let $\ORA{G}^{(t)}$ be the number of points of $\ORA{\Xi}^{(t)}$ that have merged with the $0$--indexed point of $\Xi^{(0)}$.
    Then $$\left(\frac{3}{4}\right)^t\ORA{G}^{(t)} \xrightarrow{L^p} \ORA{G}^{(\infty)},$$ for $2 \leq p < \infty$.  Moreover, the distribution of $\ORA{G}^{(\infty)}$ has an exponentially decaying tail.
\end{theorem}

The following corollary follows by coupling the objects in the proofs of Theorems~\ref{thm:gap-sequence-convergence} and~\ref{thm:cluster-size-convergence}; but it is not a direct consequence of the theorem statements themselves.
\begin{corollary} \label{cor:joint-convergence}
    The joint weak convergence holds
    $$\left(\Theta\ORA{\Xi}^{(t)}, \left(\frac{3}{4}\right)^t\ORA{G}^{(t)}\right) \Rightarrow \left(\Theta\ORA{\Xi}^{(\infty)}, \ORA{G}^{(\infty)}\right).$$
\end{corollary}

Let $\cG_+$ denote the space of non-negative integer-valued sequences indexed by $\dZ_+$. Suppose $\OLA{\eta}^{(0)} = e_0.$
Let $\OSLA{\eta}^{(t)}$ denote the $\cG_+$-valued random variable constructed from $\OLA{\eta}^{(t)}$ by removing all elements to the left of its left-most non-zero-valued element.
Denote by $\OLA{F}^{(t)}$ the random distribution function (corresponding to a finite measure supported on $(3/4)^t\cdot \dN$)
$$
 \OLA{F}^{(t)}(x) := \sum_{i = 0}^{\lfloor \left(\frac43 \right)^t x\rfloor}{\left(\frac38 \right)^t\OSLA{\eta}^{(t)}_i}. $$
Define
\begin{align*}
 \OLA{f}^{(t)}(x) := \OLA{F}^{(t)}(x) - \OLA{F}^{(t)}(x-(3/4)^t)
\end{align*}
to be the mass of $\OLA{F}^{(t)}(\cdot)$ at the point $x\in (3/4)^t\cdot \dN$.

\begin{theorem}\label{thm:weight-sum-convergence}
    The sequence of distribution functions $\OLA{F}^{(t)}$ converges weakly almost surely to $\OLA{F}^{(\infty)}$ as $t\to\infty$, where the (random) total mass of $\OLA{F}^{(\infty)}$ is distributed as $\ORA{\Xi}^{(\infty)}_1 - \ORA{\Xi}^{(\infty)}_0$, and the length of the (random) support of $\OLA{F}^{(\infty)}$ is distributed as $\ORA{G}^{(\infty)}_1$.
    
\end{theorem}

In Figure~\ref{fig:emp-distn}, we show simulations of $\OLA{F}^{(t)}$ for two independent copies of our dynamics. We observe convergence as $t$ increases. However, we cannot directly depict the convergence of the measures via their distribution functions.
\begin{figure}[htbp]
\centering
 \begin{subfigure}[t]{.45\columnwidth}
  \centering
  \includegraphics[width=\linewidth]{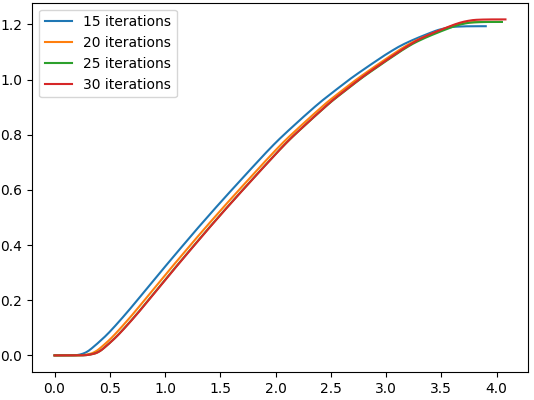}
  \caption{First Simulation}
  \label{subfig:s1}
 \end{subfigure}
 \begin{subfigure}[t]{.45\columnwidth}
  \centering
  \includegraphics[width=\linewidth]{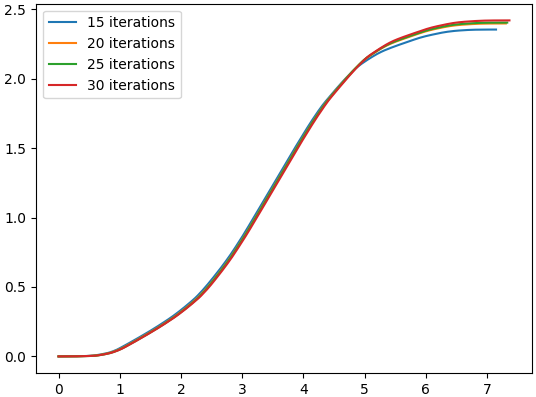}
  \caption{Second Simulation}
  \label{subfig:s2}
 \end{subfigure}
 \caption{Empirical distribution of $\OLA{F}^{(t)}$ across time for two different samples.}
 \label{fig:emp-distn}
\end{figure}

We explicitly write out a dynamics for $\OLA{F}^{(t)}$ in the following lemma that lies at the heart of creating a recursive structure for $\OLA{F}^{(t)}$.
\begin{lemma}\label{lem:rec}
    We have for all $t\geq 0, x\in\dR_+$,
 \begin{align*}
         & \OLA{F}^{(t)}(x) +
  \frac34 \sum_{n \leq (4/3)^t x}
  \OLA{f}^{(t)}((3/4)^t n) \cdot \bigl(\ind_{n \in \rho^{(t+1)}} - 1/3\bigr)                                                  \\
         & \qquad= \OLA{F}^{(t+1)}\bigl( x + (3/4)^{t+1}\cdot \bigl(| \rho^{(t+1)} \cap [0, (4/3)^t x]| - \frac13 (4/3)^t x \bigr)\bigr).
 \end{align*}
    For notational simplicity, we denote these dynamics as:
 \begin{align*}
  \OLA{F}^{(t)}(x) + \alpha^{(t+1)}_{x} = \OLA{F}^{(t+1)}(x + \beta^{(t+1)}_{x})
 \end{align*}
    where
 \begin{align*}
  \alpha^{(t+1)}_{x} & :=\frac34 \sum_{n \leq (4/3)^t x}
  \OLA{f}^{(t)}((3/4)^t n) \cdot \bigl(\ind_{n \in \rho^{(t+1)}} - 1/3\bigr),                              \\
  \beta^{(t+1)}_x    & := (3/4)^{t+1}\cdot \bigl(| \rho^{(t+1)} \cap [0, (4/3)^t x]| - \frac13 (4/3)^t x \bigr).
 \end{align*}
\end{lemma}

\subsection{Concentration for Stationary Renewal Processes}\label{sec:conc}

Here, we present several technical lemmas and definitions used to prove our main results. The proofs of these lemmas can be found in Section~\ref{sec:tech}.

Let $\rho$ be a stationary renewal process with gap distribution $Y\sim \geom(1/2) + \geom(1/2) - 1$ and let $N^{\star}_\rho(x)=\abs{\rho\cap [0,x]}, x\in\dN$ be the counting process. We have the following bound on the MGF of the excess count that will play a crucial role in the proof of Theorem~\ref{thm:weight-sum-convergence}. A more general result is given after this one.

\begin{lemma}\label{lem:renconc}
    For any $\theta_0>0$, there exists $A=A(\theta_{0})$ such that
 \begin{align*}
  \E \left[e^{-\theta(x/\mu - N^{\star}_\rho(x))}\right] &\leq  (1+A\abs{\theta})\cdot e^{A\theta^2x} \\
  \text{ and }
  \E \left[e^{-\theta(x/\mu - N^{\star}_\rho(x))}\cdot \ind_{x\in \rho}\right] &\leq \frac1\mu (1+A\abs{\theta})\cdot e^{A\theta^2x}
 \end{align*}
    for all $x\in \dN, \abs{\theta}\leq \theta_0$.
\end{lemma}

Let $Y$ be a random variable having the interoccurrence time distribution with pmf $(f_k)_{k\geq 1}$, mean $\mu=\E Y$ and mgf $\varphi(\theta):=\E e^{{\theta Y}}$ finite for some $\theta>0$. Let $Y_*$ denote a random variable with the equilibrium residual lifetime distribution with pmf
\begin{align*}
 \pr(Y^{\star}=k)=f^{\star}_k=\frac1\mu\pr(Y\geq k)\text{ for } k\geq 1.
\end{align*}
Let
\begin{align*}
    N_t   & :=\max\{k\geq 0\mid 0<Y_1+Y_2+\cdots+Y_k\leq t\},                \\
    N^{\star}_t & :=\max\{k\geq0\mid 0<Y^{\star}_1+Y_2+\cdots+Y_k\leq t\},\quad t=0,1,2,\ldots
\end{align*}
be the renewal counting process starting at $0$ and at equilibrium, respectively. We denote by $\rho$ and $\rho^{\star}$ the corresponding point processes. We assume that for all $\ga<0$,
\begin{align*}
    C_\ga:= \inf_{k\geq 1}\E(e^{\ga(Y-k)}\mid Y\geq k) >0 \text{ and }
    C^{\star}_\ga:= \inf_{k\geq 1}\E(e^{\ga(Y^{\star}-k)}\mid Y^{\star}\geq k) >0.
\end{align*}
For $\ga\geq 0$, we define $C_{\ga}=C^{\star}_{\ga}=1$.

\begin{lemma}\label{lem:conc}
    For $\theta\in\dR$, define the function $\psi(\theta):=\theta/\mu-\varphi^{-1}(e^{\theta})\geq 0$. Then, we have
 \begin{align*}
  \E e^{-\theta (N^\star_t -t/\mu) }
         & \leq \max\left(\frac{1}{C^{\star}_\ga},\frac{1}{C_\ga}\right)\cdot \E e^{\ga(Y^{\star}-1)} \cdot  e^{t \psi(\theta)} \\
  \text{ and }
  \E e^{- \theta (N^\star_t-t/\mu) }\ind_{{t\in \rho^{\star}}}
         & \leq \frac{1}{C_\ga}\cdot \frac1\mu  e^{t \psi(\theta)}
 \end{align*}
 \text{ for all } $t\geq0$ where
    $
  \ga := \varphi^{-1}(e^{\theta})
    $
    has the same sign as $\theta$.
\end{lemma}
\begin{rem}
    Using the fact that $\psi(0)=\varphi^{-1}(1)=0,\psi'(0)= 1/\mu-(\varphi^{-1})'(\varphi(0))=0, \psi''(0)= -(\varphi^{-1})''(1) = -1/\mu +\varphi''(0) \psi'(0)/\mu^2  = \var(Y)/\mu^3$, for $\abs{\theta}\ll1$ we have
    $
  \psi(\theta) \approx {\var(Y)}/{\mu^3}\cdot \theta^2/2.
    $
\end{rem}

The following lemma will be useful for our example.
\begin{lemma}\label{lem:cga}
    We have $\E e^{\ga(Y^{\star}-1)}=\frac{\E e^{\ga Y}-1}{ \mu(e^{\ga}-1)}$. Moreover, suppose that  $Y$ satisfies the following condition: the function $k\mapsto \pr(Y=k\mid Y\geq k)$ is increasing in $k$. Then $C_{\ga}=e^{-\ga}\E e^{\ga Y}$ for all $\ga<0$.
\end{lemma}

Finally, the following generalization of Lemma~\ref{lem:conc} will be used to derive exponential tail bounds for the limiting gap distribution.

\begin{lemma}\label{lem:concgen}
For $t\geq 1$ and $\mvv=(v_1,\dots,v_t)\in\dR^t$, define
$
X_t:=\sum_{k=1}^t v_k(\ind_{k\in\rho^{\star}}-1/\mu).
$
There exist constants $C,c\in(0,\infty)$,  such that for all $t\geq 1$, all $\mvv\in\dR^t$, and all $\gl\in\dR$ with
$
\abs{\gl }\cdot \norm{\mvv}_\infty < c,
$
we have
\begin{align*}
\E e^{\gl X_t}\leq e^{C\gl ^2\norm{\mvv}_2^2}.
\end{align*}
\end{lemma}

\subsection{Heuristics and a Toy Model}\label{sec:toy}
Here we consider a toy ``\emph{follow-the-leader}'' model in which only averaging occurs without any folding, and it is easy to see the asymptotic behavior of the dual weight sequence. There are infinitely many agents located at positions given by $(X_{i}(0))_{i\in\dZ}$ on the real line where $X_{i}(0)<X_{i+1}(0)$ for all $i\in\dZ$. Over time $t\in\dN$, their locations $(X_{i}(t))_{i\in\dZ}$ change in the following way.

At discrete time $t\geq 1$, a Rademacher random variable $\eps_{t}\sim \textrm{Uniform}\{-1,+1\}$ is chosen independently at random. If $\eps_{t}=1$, all points move halfway towards their right neighbor; if $\eps_{t}=-1$, all points move halfway towards their left neighbor to obtain their new positions at time $t$. Here, all points choose the same direction as the $0$-indexed point, rather than taking independent decisions as in our model (Algorithm 1). In particular, no collisions can occur, and the ordering of the points is preserved at all times. No space scaling is required to keep the intensity constant.

Let $G_{i}(t)=X_{i}(t)-X_{i-1}(t), i\in\dZ$ be the gap sequence vector between consecutive neighbors at time $t$. Note that the evolution of the gap sequence at time $t$ can be written as an action of the random linear operator $\frac12(T_0+T_{\eps_t})$ applied to the previous gap sequence $\mvG(t-1):=(G_{i}(t-1))_{i\in\dZ}$ where $T_{a}$ is the linear shift by $a$ operator given by $(T_{a}\mvx)_{i} =x_{i+a},\ i\in\dZ$. Moreover, the linear shift operators commute. Thus, we have
\begin{align*}
\mvG(t)= \prod_{i=1}^{t}\frac12(T_0+T_{\eps_i}) \cdot \mvG(0)= \frac1{2^t}\sum_{\gs_{1},\gs_{2},\ldots,\gs_{t}\in \{-1,+1\}} T_{\eps_1\gs_1 + \eps_2\gs_2+\cdots+\eps_t\gs_t} \cdot \mvG(0).
\end{align*}
Thus the gap sequence $G_{0}(t)$ can be written as the random linear combination $\sum_{k\in\dZ} a_t(k;\mveps_t)\cdot G_k(0)$ where $\mveps_{t}=(\eps_{1},\eps_{2},\ldots,\eps_{t})$ and the weights $a_t(i;\mvv)$ as a function of $\mvv=(v_{1},v_{2},\ldots,v_{t})\in\{-1,+1\}^{t}$ is given by
\begin{align*}
a_t(k;\mvv)
&= \pr(v_1\xi_1+v_2\xi_2+\cdots +v_t\xi_t = -k),\qquad \xi_1,\xi_{2},\ldots,\xi_{t}\stackrel{\textrm{i.i.d.}}{\sim} \textrm{Uniform}\{0,1\}\\
&= \pr\biggl(\sum_{i=1}^{t}v_i\cdot (1-2\xi_i)= 2k + v_{1}+v_{2}+\cdots + v_t\biggr).
\end{align*}
Clearly, the weights are zero for $\abs{k}>t$ and using the Central Limit Theorem we can easily see that the partial sum of the random weights $\sum_{i\leq k} a_t(i;\mveps_t)$ for $k\approx \theta \cdot \sqrt{t} - \frac12\sum_{i=1}^{t}\eps_{i},\ \theta\in\dR$ converges to Gaussian CDF $\Phi(2\theta)$.
Thus, the total mass of the weights is constant here, and we have a ``random'' limiting distribution for the appropriately scaled partial sum of the weight sequence given by a Normal distribution with a random shift arising out of $t^{-1/2}\sum_{i=1}^{t}\eps_{i}$.

\section{Proofs of Main Results}
\label{sec:conv-profs}
\subsection{Proof of Theorem~\ref{thm:gap-sequence-convergence}}
Recall that $\OLA{\eta}^{(t)}$ is a sequence of \emph{weights}, from which we will construct (after a re-scaling by $\left(\frac38\right)^t$) the gap sequence of the limiting point process.
Note that $\OLA{\eta}^{(t)}$ is supported on a subset of $\{-2^t, -2^t+1, \ldots, 2^t\}$, so that $\sum_{i=-\infty}^{\infty}{\OLA{\eta}^{(t)}_i} < \infty.$
Define for $t \geq 0$ the process $$\OLA{M}^{(t)} := \left(\frac38\right)^t \sum_{i=-\infty}^{\infty}{\OLA{\eta}^{(t)}_i} < \infty.$$
We can express the evolution of $\OLA{M}^{(t)}$ as follows
\begin{align}
 \OLA{M}^{(t+1)} & = \left(\frac38\right)^{t+1}\left(2\norm{\OLA{\eta}^{(t)}}_1 + 2\sum_{i \in \rho^{(t+1)}}\OLA{\eta}^{(t)}_i\right)                                 \notag \\
                    & = \frac34 \left(\OLA{M}^{(t)} + \left(\frac38\right)^t\sum_{i \in \mathrm{supp}\OLA{\eta}^{(t)}}\OLA{\eta}^{(t)}_i\ind_{i \in \rho^{(t+1)}}\right) \notag\\
                    & = \OLA{M}^{(t)} + \frac34\sum_{i \in \mathrm{supp}\OLA{\eta}^{(t)}}\OLA{X}^{(t)}_i\left(\ind_{i \in \rho^{(t+1)}} -
 \frac13\right)
    =: \OLA{M}^{(t)} + \OLA{N}^{(t+1)}.\label{eq:mn}
\end{align}
Here, $\OLA{X}^{(t)}_{\cdot}:= \left(\frac{3}{8}\right)^t\OLA{\eta}^{(t)}_{\cdot}$.
Notice that $\OLA{N}^{(t+1)}$ has conditional mean $0$.
It follows that the process $\OLA{M}^{(t)}$ is a unit-mean positive martingale; hence, it converges to an a.s.~limit $\OLA{M}^{(\infty)}$.

Notice that since $\max_{i \in \dZ}\OLA{\eta}^{(t)}_i \leq 2^t$, we have that 
$$
\max_{i \in \dZ}\left(\frac{3}{8}\right)^t\OLA{\eta}^{(t)}_i \leq \left(\frac34\right)^t.
$$
Equivalently, we have $\max_{i \in \dZ}\OLA{X}_i^{(t)} \leq \left(\frac34\right)^t$.
Thus,
\begin{align}\label{eq:xtsq}
\sum_{i\in\dZ}\left(\OLA{X}_{i}^{(t)} \right)^2= \sum_{i\in\dZ}\left(\left(\frac{3}{8}\right)^t\OLA{\eta}^{(t)}_i\right)^2 \leq \left(\frac34\right)^t\OLA{M}^{(t)},
\end{align}
 which converges a.s.~to $0$ at an exponential rate due to the a.s. convergence of $\OLA{M}^{(t)}$ to $\OLA{M}^{(\infty)}$.

We will now show the $L^2$ convergence of $\OLA{M}^{(t)}$ to $\OLA{M}^{(\infty)}$.
Indeed, notice that
\begin{align*}
 \E\left[\left(\OLA{M}^{(t+1)}\right)^2 \ \biggl|\ \cF^{(t)}\right]
     & = \left(\OLA{M}^{(t)}\right)^2 + \E\left[\left(\OLA{N}^{(t+1)}\right)^2 \ \biggl|\ \cF^{(t)}\right],
\end{align*}
since $\OLA{N}^{(t+1)}$ has conditional mean $0$.
Here, the $\sigma$-algebra $\cF^{(t)}$ contains only those events related to the dynamics of $\OLA{\eta}^{(t)}$: specifically, it does not contain events related to $\Gamma^{(0)}$.
Substituting, we have
\begin{align*}
 \E\left[\left(\OLA{N}^{(t+1)}\right)^2 \ \biggl|\ \cF^{(t)}\right]
     & \leq \left(\frac34\right)^2 \E\sum_{i \in \mathrm{supp}{\OLA{\eta}^{(t)}}}\sum_{j \in \mathrm{supp}{\OLA{\eta}^{(t)}}}\OLA{X}_i^{(t)}\OLA{X}_j^{(t)}\left(\ind_{i \in \rho^{(t+1)}} - \frac13\right)\left(\ind_{j \in \rho^{(t+1)}}-\frac13\right) \\
     & \leq \sum_{i \in \mathrm{supp}{\OLA{\eta}^{(t)}}}\sum_{j \in \mathrm{supp}{\OLA{\eta}^{(t)}}}\left(\frac34\right)^{2t}\abs{ \cov(\ind_{i \in \rho^{(t+1)}}, \ind_{j \in \rho^{(t+1)}})} \\
     & \leq \left(\frac34\right)^{2t} \sum_{i\geq 1}i\cdot \abs{\cov\left(\ind_{0 \in \rho^{(t+1)}}, \ind_{i \in \rho^{(t+1)}}\right)}\to 0.
\end{align*}
Here, the convergence occurs using the fact that the renewal distribution of $\rho$ has an exponential tail and thus, $\cov(\ind_{i \in \rho^{(t)}}, \ind_{j \in \rho^{(t)}})$ decays exponentially in $|j-i|$.
To see this, let $Y \sim \mathrm{Geom}(1/2) + \mathrm{Geom}(1/2) - 1$ denote the gap distribution of $\rho$ and let $f_Y(y) := \pr(Y = y)$.
Recall that $\mu := \E Y = 3$.
Let $u_0 = 1$ and $u_i$ denote the probability of a renewal at time $i$ given that there is a renewal at time $0$.
Then $u_n = \sum_{i=0}^{n-1}u_i f_Y(n-i)$.
We have $\cov(\ind_{i \in \rho^{(t)}}, \ind_{j \in \rho^{(t)}}) = \frac1\mu(u_{|j-i|} - \frac1\mu)$.
From here, it is easily seen that $\sum_{i\geq1}i\cdot \abs{\cov(\ind_{0 \in \rho^{(t+1)}}, \ind_{i \in \rho^{(t+1)}})}$ is summable.
It follows that
\begin{align*}
 \sup_{t \geq 0} \E \left(\OLA{M}^{(t)}\right)^2 < \infty,
\end{align*}
which also establishes the $L^2$ convergence of $\OLA{M}^{(t)}$ to a finite unit-mean limit.

\begin{figure}[htbp]
 \centering
 \includegraphics[width=\linewidth]{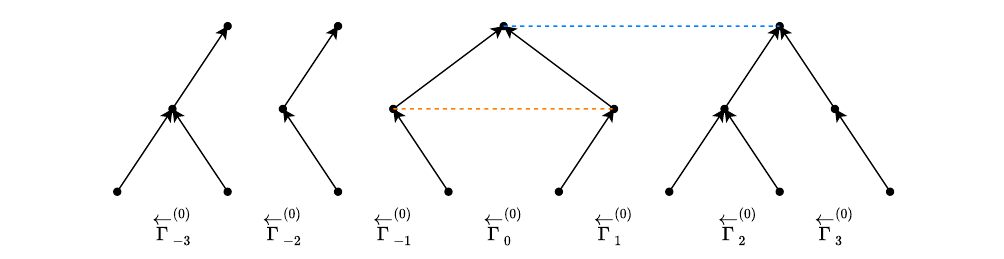}
 \caption{Relationship between the weight sequences $\eta^{(t)}$ and the initial gap sequence $\Gamma^{(0)}$. The orange dashed line, for example, is given by $\frac12\OLA{\Gamma}^{(0)}_{-1} + \OLA{\Gamma}^{(0)}_0 + \frac12\OLA{\Gamma}^{(0)}_1$; the blue dashed line is given by $\frac12\OLA{\Gamma}^{(0)}_0 + \OLA{\Gamma}^{(0)}_1 + \OLA{\Gamma}^{(0)}_2$.}
 \label{fig:inner-prod}
\end{figure}

Now, if the gap distribution of $\Xi_0$ has finite variance $\sigma^2$, we have the following (see Figure~\ref{fig:inner-prod}):
\begin{align*}
 \left\langle \left(\frac{3}{8}\right)^t\OLA{\eta}^{(t)}, \OLA{\Gamma}^{(0)}\right\rangle := \sum_{i \in \dZ}\left(\frac{3}{8}\right)^t\OLA{\eta}^{(t)}_i\OLA{\Gamma}^{(0)}_i
    &= \OLA{M}^{(t)} + \sum_{i \in \dZ}\left(\frac{3}{8}\right)^t\OLA{\eta}^{(t)}_i(\OLA{\Gamma}^{(0)}_i - 1)\\ & =: \OLA{M}^{(t)} + \OLA{L}^{(t)}.
\end{align*}

Obviously, we have $\E \OLA{L}^{(t)} = 0$ for all $t$, since $\E \Gamma^{(0)}_i = 1$ for all $i \in \dZ$.
Noting that $(\OLA{\Gamma}_i^{(0)})_{i \in \dZ}$ are i.i.d.~with variance $\sigma^2$, we have
\begin{align*}
 \E \left[\left(\OLA{L}^{(t+1)}\right)^2 \ \biggl\vert\ \cF^{(t)}\right] & = \sum_{i \in \dZ}\sum_{j \in \dZ}\left(\frac38\right)^t\OLA{\eta}_i^{(t)}\left(\frac38\right)^t\OLA{\eta}_j^{(t)}
 \cdot \cov(\OLA{\Gamma}_i^{(0)}, \OLA{\Gamma}_j^{(0)})\\
    & = \sum_{i \in \dZ}\left(\left(\frac38\right)^t\OLA{\eta}_i^{(t)}\right)^2 \sigma^2 \leq \left(\frac34\right)^t\OLA{M}^{(t)} \sigma^2 \xrightarrow{a.s.} 0;
\end{align*}
and thus $\OLA{L}^{(t)} \to 0$ in $L^2$.
It follows that $\langle \left(\frac{3}{8}\right)^t\OLA{\eta}^{(t)}, \OLA{\Gamma}^{(0)}\rangle \to \OLA{M}^{(\infty)}$ in $L^2$.
This establishes the one-dimensional distribution convergence since the choice of $\OLA{\eta}^{(0)} = e_0$ is without loss of generality.

We now extend the argument to finite-dimensional distributions. Let $F\subset \dZ$ be finite and take
\begin{align*}
\OLA{\eta}^{(0)}=\sum_{i\in F} e_i.
\end{align*}
Using the same renewal environment $(\rho^{(t)})_{t\ge 1}$, the corresponding reverse-time weight processes form a finite collection of martingales, each of which satisfies the same a.s.\ and $L^2$ convergence bounds as in the case $\OLA{\eta}^{(0)}=e_0$ (by shift-invariance of the dynamics and the previous argument). Therefore, any finite linear combination of the associated dual pairings converges in $L^2$, and hence in distribution. This yields convergence of all finite-dimensional marginals.

Finally, we show that the limiting gap distribution has exponential tails.
Notice that the limiting gap distribution is given by $\OLA{M}^{(\infty)}$, so it suffices to obtain a bound via the dynamics on $\OLA{M}$. We will prove that for some $\gl>0$, $\E e^{\gl \OLA{M}^{(\infty)}} <\infty$.

Note that the random variables $\OLA{M}_t$ are bounded for all $t\in\{1,2,\ldots\}$ and thus have finite MGFs. Using the decomposition $\OLA{M}_{t+1}-\OLA{M}_t = \OLA{N}^{(t+1)}$ from equation~\eqref{eq:mn}, the MGF bound from Lemma~\ref{lem:concgen}, and the bound from equation~\eqref{eq:xtsq}, we have for some constants $c,C>0$ and all $\abs{\gl}\leq c, t\geq 0$
\begin{align*}
\E (e^{\gl(\OLA{M}_{t+1}-\OLA{M}_t)}\mid \cF_t) \le\exp(C\cdot \gl^2\cdot (3/4)^t\OLA{M}_t).
\end{align*}
Thus, we have for $0<\gl\leq c, t\geq 0$,
\begin{align*}
\E e^{\gl \OLA{M}_{t+1}}\leq \E e^{\gl (1+C(3/4)^t\gl)\OLA{M}_t}.
\end{align*}
Now fix $\gl_{0}\in [0,c]$ and recursively define $\gl_{t+1}$ as a function of $\gl_{t}$ by
$$\gl_t=\gl_{t+1}\cdot (1+C\gl_{t+1}\cdot (3/4)^t), t\geq 0.$$ We have
\begin{align*}
\E e^{\gl _{t+1}\OLA{M}_{t+1}}\leq \E e^{\gl _{t}\OLA{M}_{t}} \text{ for all } t\geq 0.
\end{align*}
Clearly, $\gl_{t}$ is decreasing and converges to a constant $\gl_\infty\geq \gl_{0}\prod_{t=0}^{\infty}(1+C\gl_{0}(3/4)^{t})^{-1}>0$. By Fatou's lemma, we obtain a finite MGF bound for $M_\infty$, completing the proof.\qed

\subsection{Proof of Theorem~\ref{thm:cluster-size-convergence}}
The proof follows the same martingale argument as in Theorem~\ref{thm:gap-sequence-convergence}, with only a modification of the reverse-time update rule to encode cluster-size growth rather than gap mass.
Specifically, in the reverse-time construction, instead of using the tuples $(\OLA{\eta_i}^{(t)}, \OLA{\eta_i}^{(t)})$ and $(\OLA{\eta_i}^{(t)}, 2\OLA{\eta_i}^{(t)}, \OLA{\eta_i}^{(t)})$,
we instead use the tuples
$(\OLA{\eta_i}^{(t)}, \OLA{\eta_i}^{(t)})$ and $(0, 2\OLA{\eta_i}^{(t)}, 2\OLA{\eta_i}^{(t)})$. All subsequent steps (martingale decomposition, $L^2$ bounds, and convergence) are unchanged, since they depend only on the same renewal-environment covariance estimates and the same scaling structure. The exponential tail decay also follows the same argument as used in the proof of Theorem~\ref{thm:gap-sequence-convergence}.\qed
\subsection{Proof of Corollary~\ref{cor:duality}}
It suffices to verify the duality identity for initial conditions of the form $\OLA{\eta}^{(0)}=e_i$, $i\in \dZ$, since both sides are linear in $\OLA{\eta}^{(0)}$ (for finitely supported initial configurations), and the general statement follows by linearity of expectation.
Recall that, for $h(\cdot, \cdot)$ defined as the inner product in the corollary statement, the Markov processes $\ORA{\Gamma}^{(\cdot)}$ and $\OLA{\eta}^{(\cdot)}$ are dual if for all $t$, $\ORA{\Gamma}^{(0)}$, and $\OLA{\eta}^{(0)}$, we have
$$
\E_{\ORA{\Gamma}^{(0)}}h( \ORA{\Gamma}^{(t)}, \OLA{\eta}^{(0)}) = \E_{\OLA{\eta}^{(0)}}h(\ORA{\Gamma}^{(0)}, \OLA{\eta}^{(t)}).
$$
For $\OLA{\eta}^{(0)} = e_i$, the result is shown in the proof of Theorem~\ref{thm:gap-sequence-convergence}.\qed

\subsection{Proof of Corollary~\ref{cor:joint-convergence}}
This result follows by taking the proofs of Theorems~\ref{thm:gap-sequence-convergence} and~\ref{thm:cluster-size-convergence}, and coupling the point processes used to determine the un-merging indices.\qed


\subsection{Proof of Theorem~\ref{thm:weight-sum-convergence}}
We prove convergence of $\OLA{F}^{(t)}$ by first establishing control of the moment generating functions $G^{(t)}$, and then identifying the limiting distribution function through the recursive dynamics in Lemma~\ref{lem:rec}.
Indeed, define for $s\in\dR$,
$$
    G^{(t)}(s) := \int_{\dR}e^{s x}d\OLA{F}^{(t)}(x).
$$
Using Lemma~\ref{lem:rec}, we get:
\begin{align*}
    G^{(t+1)}(s) 
    & = \int_{\dR}e^{s x}d\OLA{F}^{(t+1)}(x)
 \\
    & = \int_{\dR}e^{s (x + \beta^{(t+1)}_{x})}d\OLA{F}^{(t+1)}(x+ \beta^{(t+1)}_{x})
 \\
    & = \int_{\dR}e^{sx +s \beta^{(t+1)}_{x}}d\OLA{F}^{(t)}(x) + \int_{\dR}e^{sx +s \beta^{(t+1)}_{x}}d\alpha^{(t+1)}_{x}\\
    &=: H_\beta^{(t+1)}(s) + H_\alpha^{(t+1)}(s),
\end{align*}
where we treat $\alpha^{(t)}_x$ as a signed measure.
Here, the second equality follows because the mapping $x \mapsto x + \beta^{(t+1)}_{x}$ is monotone increasing in $x$: we have $x + \beta^{(t+1)}_{x} = \frac{2}{3}x + N_{\rho^{(t+1)}}(x)$, where $N_{\rho^{(t+1)}}(\cdot)$ accounts for any positive jumps due to $\rho^{(t+1)}$.

Observe that for $x \in \dR$,
\begin{align*}
 \alpha_x^{(t+1)}
     & = -\frac14 \OLA{F}^{(t)}(x) + \frac34\sum_{n \leq (4/3)^t x} \OLA{f}^{(t)}((3/4)^t n)
 \cdot \ind_{n \in \rho^{(t+1)}},
\end{align*}
and that
\begin{align*}
    H_\beta^{(t+1)}(s) &:= \int_{\dR}e^{sx + s\beta_x^{(t+1)}}d\OLA{F}^{(t)}(x)
    = \sum_{n \in \dN}\exp\left(\left(\frac34\right)^tsn + s \beta^{(t+1)}_{(3/4)^tn}\right)\OLA{f}^{(t)}\left(\left(\frac34\right)^t n\right),
\end{align*}
and thus we can re-express:
\begin{align*}
    G^{(t+1)}(s) & = \frac34 \sum_{n \in \dN}\exp\left(\left(\frac34\right)^tsn + s \beta^{(t+1)}_{(3/4)^tn}\right)\OLA{f}^{(t)}\left(\left(\frac34\right)^t n\right)
 \\
    &\qquad + \frac34 \sum_{n \in \dN}\exp\left(\left(\frac34\right)^tsn + s \beta^{(t+1)}_{(3/4)^tn}\right)\OLA{f}^{(t)}\left(\left(\frac34\right)^t n\right)\ind_{n \in \rho^{(t+1)}}.
\end{align*}
For notational cleanliness, let $\hat{n} := \left(3/4\right)^t n$. 
Conditioning on $\cF_t$, the terms $\beta^{(t+1)}_{\hat n}$ are functions of the renewal process $\rho^{(t+1)}$ only, so Lemma~\ref{lem:renconc} applies termwise to the exponential factors involving $\beta^{(t+1)}_{\hat n}$.
Using Lemma~\ref{lem:renconc}, we have
\begin{align*}
 \E\bigl( G^{(t+1)}(s)\mid \cF_t\bigr) & = \frac34 \sum_{n \in \dN}\E\biggl(\exp\left(s\hat{n} + s \beta^{(t+1)}_{\hat{n}}\right) \mid \cF_t\biggr)\OLA{f}^{(t)}\left(\hat{n}\right)
 \\
    &\qquad + \frac34 \sum_{n \in \dN}\E\biggl(\exp\left(s\hat{n} + s \beta^{(t+1)}_{\hat{n}}\right)\ind_{n \in \rho^{(t+1)}}\mid \cF_t\biggr)\cdot \OLA{f}^{(t)}\left(\hat{n}\right) .
 \\
    &\leq \frac34 \sum_{n \in \dN}e^{s\hat{n}}\OLA{f}^{(t)}\left(\hat{n}\right)\left(1 + \left(\frac34\right)^{t}As\right)e^{As\hat{n}}
 \\
    &\qquad + \frac34 \sum_{n \in \dN}\frac{1}{3}e^{s\hat{n}}\OLA{f}^{(t)}\left(\hat{n}\right)\left(1 + \left(\frac34\right)^{t}As\right)e^{As\hat{n}},
\end{align*}
where the last inequality holds for all sufficiently small $s > 0$.
Finally, using Lemma~\ref{lem:rec}, we get
\begin{equation}
 \E\bigl( G^{(t+1)}(s)\mid \cF_t\bigr) \leq (1+A(3/4)^t s)\cdot G^{(t)}(s(1+A(3/4)^ts)).
 \label{eq:MGF-bound}
\end{equation}

Now take $\theta_0\in(-1/A,\infty)$. Consider the recursion
\begin{align*}
 \theta_t = \theta_{t+1}(1+A(3/4)^t \theta_{t+1}),\qquad t\geq 0.
\end{align*}
Clearly, $\theta_t$ is decreasing in $t$ and is lower bounded by $\theta_0\prod_{i=0}^\infty (1+A(3/4)^i\theta_0)^{-1}.$ Thus $\theta_*:=\lim_{t\to\infty} \theta_t$ exists. By continuity, the map $\theta_0\mapsto \theta_*$ is a bijection.
By construction of $(\theta_t)$ and the conditional estimate~\eqref{eq:MGF-bound}, the adapted process
\begin{align*}
\bigl(\theta_t\,G^{(t)}(\theta_t)\bigr)_{t\ge 0}
\end{align*}
is a nonnegative supermartingale. Hence, it converges almost surely to a finite limit, which we denote by
\begin{align*}
G^{(\infty)}(\theta_*) := \lim_{t\to\infty} G^{(t)}(\theta_t).
\end{align*}

To identify the limit, we use the previously established convergence of the scaled gap distributions together with the a.s.\ control on total mass and support length encoded by the reverse-time construction. In particular, on an event of probability one, the family $\{\OLA{F}^{(t)}\}_{t\ge 0}$ is tight and has uniformly bounded total mass and support along large times. Consequently, for any sequence $t_k\to\infty$, there exists a further subsequence (not relabeled) such that $\OLA{F}^{(t_k)}$ converges weakly to some random limit $\widetilde{F}^{(\infty)}$.

By Theorem~\ref{thm:gap-sequence-convergence}, the limiting gap law is unique; therefore, the total mass of $\widetilde{F}^{(\infty)}$ must agree in distribution with that of $\OLA{F}^{(\infty)}$, and the recursive construction forces the same identification for the full distribution function. Hence, every subsequential weak limit coincides (almost surely) with $\OLA{F}^{(\infty)}$. It follows that the full sequence $\OLA{F}^{(t)}$ converges weakly almost surely to $\OLA{F}^{(\infty)}$. \qed
%
%

\section{Proofs of Technical Results}\label{sec:tech}

\subsection{Proof of Lemma~\ref{lem:rec}}
Fix $t$ and consider some fixed integer $n \in \left[0, \left(4/3\right)^t\right]$.
There is mass $\OLA{f}^{(t)}((3/4)^t n)$ at $n$. By rescaling by a factor of $(\frac43)^t$, any new mass has the same units over time.
If there is an un-merging at the $n$-th index, then an additional $\OLA{f}^{(t)}((3/4)^t n)$ is created, and no new mass is created otherwise.
However, any new mass is also shrunk by a factor of $\frac34$ at time step $t+1$.
Thus, the contribution of the mass at index $n$ at time $(t+1)$ is given by $\left(\frac34\right)^{t+1}\cdot \OLA{f}^{(t)}((3/4)^t n)\cdot (\ind_{n \in \OLA{\rho}^{(t+1)}} - 1/3)$.
The total additional mass at time $(t+1)$ contributed from indices contributing to $\OLA{F}^{(t)}(x)$ is then given by summing over $n$.

By a similar re-scaling argument, there are $|\OLA{\rho}^{t+1} \cap [0, (4/3)^tx]|$ un-mergings in this interval, and once again, those are scaled by an additional factor of $\frac34$ at time $t+1$.

The right-hand side accounts for the contribution of new mass within the interval $[0, (4/3)^t x]$ at time $t$ to the distribution function $F^{(t+1)}(\cdot)$.
We need to account for the extra mass created at time $t$; this results in a change to the argument of $F^{(t+1)}(\cdot)$.
To do this, we count the number of gaps that occur in the interval $[0, (4/3)^t x]$.
This increases by one at each point of $\OLA{\rho}^{t+1} \cap [0, (\frac43)^tx]$, but again any new mass is shrunken by a factor of $\frac34$ at time step $t+1$.
Thus, the total number of gaps, after re-scaling, is given by $x + \beta^{(t+1)}_x$.
The result follows.\qed

\subsection{Proof of Lemma~\ref{lem:renconc}}
We will use Lemma~\ref{lem:conc}.
Let $Y\sim \geom(1/2) + \geom(1/2) - 1$ be the renewal gap distribution with pmf
$
    f_k=k\cdot2^{-k-1} \text{ for } k\in\dN,
$
mean $\mu=3$, and mgf
\begin{align*}
 \varphi(\theta)=\E e^{\theta Y} =  e^\theta (2-e^\theta )^{-2} \geq e^{\mu\theta} \text{ for } \theta<\ln2.
\end{align*}
Let $Y_*$ denote a random variable with the equilibrium residual lifetime distribution with pmf
\begin{align*}
    f^{\star}_k=\frac1\mu\pr(Y\geq k)=\frac13(k+1)\cdot2^{-k}\text{ for } k\geq 1.
\end{align*}
It is easy to check that both $Y$ and $Y^{\star}$ satisfy the condition in Lemma~\ref{lem:cga}. Thus, we have $C_{\ga}=\ind_{\ga\geq 0} + e^{\theta-\ga}\ind_{\ga<0}$ when $\ga=\varphi^{-1}(e^{\theta})\leq \theta/\mu$.  Similarly, $C^{\star}_{\ga}=\ind_{\ga\geq 0} + \E e^{\ga(Y^{\star}-1)}\ind_{\ga<0}$.

Next, we solve for $\varphi(x)=e^{\theta}$, to get
\begin{align*}
    e^{2x}-(4+e^{-\theta})e^x+4 = 0 \text{ or } e^x = \frac{4+e^{-\theta}- \sqrt{e^{-\theta}(e^{-\theta}+8)}}{2}.
\end{align*}
Thus, for $\theta\in\dR$, we have
\begin{align*}
 \psi(\theta)
     & = \theta/3 - \varphi^{-1}(e^{\theta})\\
     & = \theta/3-\log\left(1- \frac{2(e^{-\theta}-1)}{2+e^{-\theta}+ \sqrt{e^{-\theta}(e^{-\theta}+8)}} \right) \\
     & \leq \theta/3-\log\left(1- \frac{e^{-\theta}-1}{2+e^{-\theta}} \right)
 \leq \log \left(\frac23e^{\theta/3}+\frac13e^{-2\theta/3}\right)
 \leq \frac18\theta^2.
\end{align*}
Using Lemma~\ref{lem:conc}, we get
\begin{align*}
 \E e^{\theta(x/\mu- N_\rho^\star(x))} &\leq (1+A\abs{\theta})\cdot   e^{x\theta^2/8}\\
 \text{ and }
 \E e^{\theta(x/\mu- N_\rho^\star(x))} \ind_{{x\in \rho}}&\leq e^{\max(0,-\theta)}\cdot \frac1\mu e^{x\theta^2/8}
\end{align*}
$\text{ for all } x\geq1, \theta\in\dR.$
This proves the lemma.\qed

\subsection{Proof of Lemma~\ref{lem:conc}}
Fix $\theta\in \dR$, and define
$
    m_t:=\E e^{-\theta N_t}, m^{\star}_t:=\E e^{-\theta N^{\star}_t} \text{ for } t\geq 0.
$
We have
\begin{align*}
    m_t = \pr(Y>t)+ e^{-\theta}\sum_{k=1}^t f_km_{t-k}.
\end{align*}
Define
\begin{align*}
 \ga:=\varphi^{-1}(e^{\theta})\leq \theta/\mu
\end{align*}
so that $\varphi(\ga)=e^{\theta}$. The last inequality follows from Jensen's inequality. Note that $\ga>0$ for $\theta>0$ and $\ga<0$ for $\theta<0$.
Define $\olm_t:= e^{\ga t}m_t.$
Let $f^{(\ga)}$ be the pmf given by $f^{(\ga)}_k:=e^{\ga k -\theta}f_k, k\geq 1$ and $Y^{(\ga)}$ be a random variable with pmf $f^{(\ga)}$.
We get
\begin{align*}
 \olm_t = \mu e^{\ga t}f_{t+1}^{\star}+ \sum_{k=1}^t f^{(\ga)}_k \cdot \olm_{t-k}.
\end{align*}
Let
\begin{align*}
    A(\ga):=\sum_{t=0}^{\infty} e^{\ga t}f_{t+1}^{\star} = e^{-\ga}\E e^{{\ga Y^{\star}}}
\end{align*}
and $Z_{\ga}\equald Y^{*,(\ga)}-1$ be a random variable with pmf $e^{\ga t}f_{t+1}^{\star} /A(\ga)$, $t\geq 0$.
Thus, we have
\begin{align*}
 \olm_t
     & = \mu A(\ga)\cdot \pr(Y^{*,(\ga)}+Y_1^{(\ga)}+Y_2^{(\ga)}+\cdots+Y_k^{(\ga)}=t+1\text{ for some } k\geq 0 )                                            \\
     & \leq \frac{1}{C_{\ga}}\cdot e^{\theta-\ga}\cdot \E Y^{(\ga)} \cdot \pr(Y^{(\ga),*}+Y_1^{(\ga)}+Y_2^{(\ga)}+\cdots+Y_k^{(\ga)}=t+1\text{ for some } k\geq 0 ) \\
     & \leq \frac{1}{C_{\ga}}\cdot e^{\theta-\ga}.
\end{align*}
Here, we used the fact that for all $k\geq 1$,
\begin{align*}
 \frac{\mu  A(\ga) \cdot \pr(Y^{*,(\ga)}= k)}{\E Y^{(\ga)}\cdot \pr(Y^{(\ga),*}= k) }
    = \frac{e^{\theta-\ga} \pr(Y\geq k)}{\sum_{t\geq k} e^{\ga (t-k)} f_t} \leq e^{\theta-\ga} \cdot \frac{1}{C_\ga}.
\end{align*}
Similarly, for the equilibrium renewal process, we get
$
    m^{\star}_t  = \pr(Y^{\star}> t) + e^{-\theta}\sum_{k=1}^t f^{\star}_k m_{t-k} $, or
\begin{align*}
    e^{\ga t} m^{\star}_t
     & = e^{\ga t}\pr(Y^{\star}> t) +  e^{-\theta}\sum_{k=1}^t e^{\ga k}f^{\star}_k\cdot \olm_{t-k}\\
     & \leq e^{-\ga}\cdot e^{\ga (t+1)}\pr(Y^{\star}\geq t+1) + \frac{1}{C_\ga}\cdot e^{-\ga} \sum_{k=1}^t e^{\ga k}f^{\star}_k \\
     &\leq \max\left(\frac{1}{C^{\star}_\ga},\frac{1}{C_\ga}\right)\cdot \E e^{\ga(Y^{\star}-1)}.
\end{align*}
For the second part, define
\begin{align*}
    n_{t}:=\mu e^{\ga t}\E( e^{\theta N^\star_t}\cdot \ind_{t\in \rho^{\star}}), t\geq 1,
\end{align*}
which satisfies the recursion
\begin{align*}
    n_{t} & = \mu e^{-\theta}e^{\ga t}f_{t}^{\star} + \sum_{k=1}^{t} f^{(\ga)}_{k}\cdot n_{t-k}.
\end{align*}
This is bounded by $1/C_\ga$ by the same argument used in the first part. Simplifying, we get the result.\qed

\subsection{Proof of Lemma~\ref{lem:cga}}
The first equality follows from the fact that $\pr(Y^{\star}=k)=\frac1\mu \pr(Y\geq k), k\geq 1$ and by changing the order of the sums. For the second result, we will show that $\E(e^{\ga(Y-k)}\mid Y\geq k)$ is increasing in $k\geq 1$, when $\ga<0$. Then the result follows by taking $k=1$. Finally, we have
\begin{align}\label{eq:prob}
     e^{\ga (k+1)}\pr(Y\geq k)\pr(Y\geq k+1)\cdot &\left( \E(e^{\ga(Y-k-1)}\mid Y\geq k+1) - \E(e^{\ga(Y-k)}\mid Y\geq k) \right) \\
     & = \sum_{i=k}^{\infty} f_i\cdot \sum_{j=k+1}^{\infty} e^{\ga j}f_j
    -
    e^{\ga} \sum_{i=k+1}^{\infty} f_i \cdot \sum_{j=k}^{\infty} e^{\ga j}f_j.\notag 
\end{align}
Simplifying, we get    
\begin{align*}
     \eqref{eq:prob}& = (1-e^{\ga}) \cdot \sum_{i=k}^{\infty}f_i \cdot \sum_{j=k}^{\infty} e^{\ga j} f_j
    - f_{k}\sum_{i=k}^{\infty} f_i (1-e^\ga)\sum_{j=k}^{i}e^{\ga j}                                                      \\
     & = (1-e^{\ga})\cdot \sum_{j=k}^{\infty} e^{\ga j} \left(f_j \pr(Y\geq k) - f_k \pr(Y\geq j)\right) >0.
\end{align*}
This completes the proof.\qed

\subsection{Proof of Lemma~\ref{lem:concgen}}
Define $\xi_k:=\ind_{k\in\rho^{\star}}-1/\mu,$ so that $(\xi_k)_{k\in\dZ}$ is stationary, centered, and bounded by $|\xi_k|\leq 1$. Recall that $Y_*$, the equilibrium residual lifetime distribution, satisfies $\pr(Y_{*}\geq k)\leq c_0 e^{-\theta k}$, $k\geq 1$, for some $\theta>0$. For completeness, we present a proof sketch based on exponential mixing of the random variables $(\xi_{k})$ and bounds on the joint cumulants. Note that all of the tools used are available in the literature. One can also derive a similar result using block decomposition.

Using the regenerative property of the renewal process, the occurrence of a renewal in a separating interval cuts the dependence between the past and the future. It follows that $(\xi_k)$ is exponentially $\beta$-mixing, \ie\
\begin{align*}
\beta(r)\leq C_1 e^{-c_1 r},\quad r\geq 1,
\end{align*}
for some constants $C_1,c_1>0$ depending only on the law of $Y$ (see, \eg, the mixing background in~\cite{rb05} and regeneration coupling arguments in~\cite{lin02}).

Next, standard cumulant estimates for bounded exponentially mixing sequences (cf.~\cite{b01book,dn07}) imply that there exist constants $A, B>0$ such that, for every $m\geq 2$ and every $i_1,\ldots,i_m\in\dZ$,
\begin{align*}
    \big|\kappa(\xi_{i_1},\ldots,\xi_{i_m})\big|
    \le A^{m-1}(m-1)!\,e^{-B\, T(i_1,\ldots,i_m)},
\end{align*}
where $\kappa(\cdot)$ denotes joint cumulant and $T(i_1,\dots,i_m)$ is the minimal spanning-tree length of $\{i_1,\ldots,i_m\}$ in the graph metric on $\dZ$.
By the multi-linearity of cumulants,
\begin{align*}
\kappa_m(X_t)
=
\sum_{i_1,\dots,i_m=1}^t
v_{i_1}\cdots v_{i_m}\,
\kappa(\xi_{i_1},\ldots,\xi_{i_m}),
\qquad m\geq 1,
\end{align*}
and $\kappa_1(X_t)=\E X_t=0$. Combining the preceding bound with the exponential summability of the kernel $e^{-B|i-j|}$ yields
\begin{align*}
|\kappa_m(X_t)|
\le
K^{m-1} m!\norm{\mvv}_2^2\norm{\mvv}_\infty^{m-2},
\qquad m\geq 2,
\end{align*}
for $K=4A\sum_{r=0}^{\infty}e^{-Br}$.
Finally, for $|\gl|\norm{\mvv}_\infty<(2K)^{-1}$, the cumulant expansion of the log-moment generating function converges absolutely
\begin{align*}
\log \E e^{\gl X_t}
=
\sum_{m\geq 2}\frac{\kappa_m(X_t)}{m!}\gl ^m.
\end{align*}
Hence
\begin{align*}
\log \E e^{\gl X_t}
\le
K \norm{\mvv}_2^2 \sum_{m\geq 2} (K|\gl |\norm{\mvv}_\infty)^{m-2}\gl ^2
=2K\gl ^2\norm{\mvv}_2^2.
\end{align*}
Renaming constants gives the claimed MGF bound. \qed

\section{Future Work}\label{sec:future-work}
In this paper, we analyze a simple dynamical clustering algorithm for $\dR$-valued data in which each data point attempts to cluster with either its left or right neighbor, chosen uniformly at random. For this model, we establish (under mild assumptions) a unique scaling limit that is independent of the initial data.

We conjecture that (under similar assumptions), a similar scaling limit holds for other realistic dynamic clustering algorithms, such as Algorithm 2 from the introduction of this paper.
For this algorithm, the scaling limit appears to depend on the initial data.
The proof technique used in this work does not apply to Algorithm 2, since the weight sequences under time reversal are no longer integer-valued.
Determining the existence of stationary distributions for this type of dynamics remains an important area of future work. New techniques are required to study these algorithms with greater generality.

Also, an open problem is to identify the limiting non-renewal point process as the unique fixed point of an appropriate stochastic dynamics. Based on simulations, it seems that the weak limit (in the local weak sense) of the genealogy tree exists. Another relevant question is whether we can embed it in a hyperbolic space and obtain a space-time limit that includes both the spatial component and the time dynamics.

\bibliographystyle{alpha}
\bibliography{references}

@article{li2021survey,
  title={A SURVEY OF CLUSTERING METHODS VIA OPTIMIZATION METHODOLOGY.},
  author={Li, Xiaotian and Cai, Linju and Li, Jingchao and Yu, Carisa Kwok Wai and Hu, Yaohua},
  journal={Journal of Applied \& Numerical Optimization},
  volume={3},
  number={1},
  year={2021}
}

@article{khaniha2025hierarchical,
  title={Hierarchical Clustering Algorithms on {P}oisson and {C}ox Point Processes},
  author={Khaniha, Sayeh and Baccelli, Fran{\c{c}}ois},
  journal={arXiv preprint arXiv:2503.18555},
  year={2025}
}

@article{angel2023tale,
  title={A tale of two balloons},
  author={Angel, Omer and Ray, Gourab and Spinka, Yinon},
  journal={Probability Theory and Related Fields},
  volume={185},
  number={3},
  pages={815--837},
  year={2023},
  publisher={Springer}
}

@article{hendrickx2012convergence,
  title={Convergence of type-symmetric and cut-balanced consensus seeking systems},
  author={Hendrickx, Julien M and Tsitsiklis, John N},
  journal={IEEE Transactions on Automatic Control},
  volume={58},
  number={1},
  pages={214--218},
  year={2012},
  publisher={IEEE}
}

@article{friedkin1999social,
  title={Social influence networks and opinion change},
  author={Friedkin, Noah E. and Johnsen, Eugene C.},
  journal={Advances in Group Processes},
  volume={16},
  number={1},
  pages={1--29},
  year={1999}
}

@article{degroot1974reaching,
  title={Reaching a consensus},
  author={DeGroot, Morris H.},
  journal={Journal of the American Statistical Association},
  volume={69},
  number={345},
  pages={118--121},
  year={1974}
}

@article{hegselmann2002opinion,
  title={Opinion dynamics and bounded confidence models, analysis, and simulation},
  author={Hegselmann, Rainer and Krause, Ulrich},
  journal={Journal of Artificial Societies and Social Simulation},
  volume={5},
  pages={1--33},
  year={2002}
}

@book{bullo2009distributed,
  title={Distributed Control of Robotic Networks: {A} Mathematical Approach to Motion Coordination Algorithms},
  author={Bullo, Francesco and Cort{\'e}s, Jorge and Mart{\'\i}nez, Sonia},
  publisher={Princeton University Press},
  year={2009}
}

@article{jadbabaie2003coordination,
  title={Coordination of groups of mobile autonomous agents using nearest neighbor rules},
  author={Jadbabaie, Ali and Lin, Jie and Morse, A. Stephen},
  journal={IEEE Transactions on Automatic Control},
  volume={48},
  number={6},
  pages={988--1001},
  year={2003},
  month={June}
}

@article{chazelle2014convergence,
  title={The convergence of bird flocking},
  author={Chazelle, Bernard},
  journal={Journal of the ACM},
  volume={61},
  number={4},
  pages={1--35},
  year={2014}
}

@article{bernardo2024bounded,
  title={Bounded confidence opinion dynamics: {A} survey},
  author={Bernardo, Carmela and Altafini, Claudio and Proskurnikov, Anton and Vasca, Francesco},
  journal={Automatica},
  volume={159},
  pages={111302},
  year={2024},
  publisher={Elsevier}
}

@inproceedings{aggarwal2003framework,
  title={A framework for clustering evolving data streams},
  author={Aggarwal, Charu C and Philip, S Yu and Han, Jiawei and Wang, Jianyong},
  booktitle={Proceedings 2003 VLDB conference},
  pages={81--92},
  year={2003},
  organization={Elsevier}
}

@article{delvenne2010stability,
  title={Stability of graph communities across time scales},
  author={Delvenne, Jean-Charles and Yaliraki, Sophia N and Barahona, Mauricio},
  journal={Proceedings of the National Academy of Sciences},
  volume={107},
  number={29},
  pages={12755--12760},
  year={2010},
  publisher={National Academy of Sciences}
}

@article{Bhatti2019,
  author    = {Shahzad F.~Bhatti and Carolyn L.~Beck and Angelia Nedi\'c},
  title     = {Data Clustering and Graph Partitioning via Simulated Mixing},
  journal   = {IEEE Transactions on Network Science and Engineering},
  year      = {2019},
  volume    = {6},
  number    = {3},
  pages     = {253--266},
  doi       = {10.1109/TNSE.2018.2869110}
}

@article{aldous1999deterministic,
 ISSN = {13507265},
 URL = {http://www.jstor.org/stable/3318611},
 author = {David J. Aldous},
 journal = {Bernoulli},
 number = {1},
 pages = {3--48},
 publisher = {International Statistical Institute (ISI) and Bernoulli Society for Mathematical Statistics and Probability},
 title = {Deterministic and Stochastic Models for Coalescence (Aggregation and Coagulation): A Review of the Mean-Field Theory for Probabilists},
 urldate = {2026-02-19},
 volume = {5},
 year = {1999}
}

@article{xu2005survey,
  title={Survey of clustering algorithms},
  author={Xu, Rui and Wunsch, Donald},
  journal={IEEE Transactions on Neural Networks},
  volume={16},
  number={3},
  pages={645--678},
  year={2005},
  publisher={IEEE}
}

@article{drezner1996facility,
  title={Facility location: {A} survey of applications and methods},
  author={Drezner, E},
  journal={Journal of the Operational Research Society},
  volume={47},
  number={11},
  pages={1421--1421},
  year={1996},
  publisher={Springer}
}

@article{likas2003global,
  title={The global $k$-means clustering algorithm},
  author={Likas, Aristidis and Vlassis, Nikos and Verbeek, Jakob J},
  journal={Pattern Recognition},
  volume={36},
  number={2},
  pages={451--461},
  year={2003},
  publisher={Elsevier}
}

@book{gersho2012vector,
  title={Vector Quantization and Signal Compression},
  author={Gersho, Allen and Gray, Robert M},
  volume={159},
  year={2012},
  publisher={Springer Science \& Business Media}
}

@article{zhang2008continuous,
  title={Continuous $k$-means monitoring over moving objects},
  author={Zhang, Zhenjie and Yang, Yin and Tung, Anthony KH and Papadias, Dimitris},
  journal={IEEE Transactions on Knowledge and Data Engineering},
  volume={20},
  number={9},
  pages={1205--1216},
  year={2008},
  publisher={IEEE}
}

@inproceedings{ghanbari2011tracking,
  title={Tracking adaptive performance models using dynamic clustering of user classes},
  author={Ghanbari, Hamoun and Barna, Cornel and Litoiu, Marin and Woodside, Murray and Zheng, Tao and Wong, Johnny and Iszlai, Gabriel},
  booktitle={Proceedings of the 2nd ACM/SPEC International Conference on Performance Engineering},
  pages={179--188},
  year={2011}
}

@article{cortes2004coverage,
  title={Coverage control for mobile sensing networks},
  author={Cortes, Jorge and Martinez, Sonia and Karatas, Timur and Bullo, Francesco},
  journal={IEEE Transactions on Robotics and Automation},
  volume={20},
  number={2},
  pages={243--255},
  year={2004},
  publisher={IEEE}
}

@inproceedings{frazzoli2004decentralized,
  title={Decentralized algorithms for vehicle routing in a stochastic time-varying environment},
  author={Frazzoli, Emilio and Bullo, Francesco},
  booktitle={43rd IEEE Conference on Decision and Control (CDC)},
  volume={4},
  pages={3357--3363},
  year={2004},
  organization={IEEE}
}

@inproceedings{etesami2024distributed,
  title={Distributed Data Placement and Content Delivery in Web Caches with Non-Metric Access Costs},
  author={Etesami, S. Rasoul},
  booktitle={Proceedings of the ACM Web Conference},
  pages={4340--4351},
  year={2024}
}

@article{sharma2011entropy,
  title={Entropy-based framework for dynamic coverage and clustering problems},
  author={Sharma, Puneet and Salapaka, Srinivasa M and Beck, Carolyn L},
  journal={IEEE Transactions on Automatic Control},
  volume={57},
  number={1},
  pages={135--150},
  year={2011},
  publisher={IEEE}
}

@article{xu2013clustering,
  title={Clustering and coverage control for systems with acceleration-driven dynamics},
  author={Xu, Yunwen and Salapaka, Srinivasa M and Beck, Carolyn L},
  journal={IEEE Transactions on Automatic Control},
  volume={59},
  number={5},
  pages={1342--1347},
  year={2013},
  publisher={IEEE}
}

@book{nunez2006electric,
  title={Electric {F}ields of the {B}rain: {T}he {N}europhysics of {EEG}},
  author={Nunez, Paul L and Srinivasan, Ramesh},
  year={2006},
  publisher={Oxford University Press}
}

@article{sheu2002fuzzy,
  title={A fuzzy clustering-based approach to automatic freeway incident detection and characterization},
  author={Sheu, Jiuh-Biing},
  journal={Fuzzy Sets and Systems},
  volume={128},
  number={3},
  pages={377--388},
  year={2002},
  publisher={Elsevier}
}

@article{altafini2013consensus,
  title={Consensus problems on networks with antagonistic interactions},
  author={Altafini, Claudio},
  journal={IEEE Transactions on Automatic Control},
  volume={58},
  number={4},
  pages={935--946},
  year={2013},
  publisher={IEEE}
}

@article{wang2017noisy,
  title={Noisy Hegselmann-Krause systems: {P}hase transition and the $2R$-conjecture},
  author={Wang, Chu and Li, Qianxiao and E, Weinan and Chazelle, Bernard},
  journal={Journal of Statistical Physics},
  volume={166},
  number={5},
  pages={1209--1225},
  year={2017},
  publisher={Springer}
}

@article{dey20252r,
  title={The $2R$-Conjecture for the {H}egselmann--{K}rause Model: {A} Proof in Expectation and New Directions},
  author={Dey, Partha S and Etesami, S. Rasoul and Gopalan, Aditya S},
  journal={arXiv preprint arXiv:2508.08299},
  year={2025}
}

@article{al12,
author = {David Aldous and Daniel Lanoue},
title = {{A lecture on the averaging process}},
volume = {9},
journal = {Probability Surveys},
number = {none},
publisher = {Institute of Mathematical Statistics and Bernoulli Society},
pages = {90 -- 102},
year = {2012},
doi = {10.1214/11-PS184},
URL = {https://doi.org/10.1214/11-PS184}
}

@article{cdsz22,
author = {Sourav Chatterjee and Persi Diaconis and Allan Sly and Lingfu Zhang},
title = {{A phase transition for repeated averages}},
volume = {50},
journal = {The Annals of Probability},
number = {1},
publisher = {Institute of Mathematical Statistics},
pages = {1 -- 17},
year = {2022},
doi = {10.1214/21-AOP1526},
URL = {https://doi.org/10.1214/21-AOP1526}
}

@article{spiro22, 
title={An averaging process on hypergraphs}, volume={59}, DOI={10.1017/jpr.2021.67}, number={2}, journal={Journal of Applied Probability}, author={Spiro, Sam}, year={2022}, pages={495–504}}

@article {k73,
    AUTHOR = {Kallenberg, Olav},
     TITLE = {Characterization and convergence of random measures and point
              processes},
   JOURNAL = {Z. Wahrscheinlichkeitstheorie und Verw. Gebiete},
  FJOURNAL = {Zeitschrift f\"ur Wahrscheinlichkeitstheorie und Verwandte
              Gebiete},
    VOLUME = {27},
      YEAR = {1973},
     PAGES = {9--21}
}

@article {rb05,
    AUTHOR = {Bradley, Richard C.},
     TITLE = {Basic properties of strong mixing conditions. {A} survey and
              some open questions},
      NOTE = {Update of, and a supplement to, the 1986 original},
   JOURNAL = {Probab. Surv.},
  FJOURNAL = {Probability Surveys},
    VOLUME = {2},
      YEAR = {2005},
     PAGES = {107--144},
      ISSN = {1549-5787},
   MRCLASS = {60G10},
  MRNUMBER = {2178042},
MRREVIEWER = {Lajos\ Horv\'ath},
       DOI = {10.1214/154957805100000104},
       URL = {https://doi-org.proxy2.library.illinois.edu/10.1214/154957805100000104},
}

@book {lin02,
    AUTHOR = {Lindvall, Torgny},
     TITLE = {Lectures on the coupling method},
      NOTE = {Corrected reprint of the 1992 original},
 PUBLISHER = {Dover Publications, Inc., Mineola, NY},
      YEAR = {2002},
     PAGES = {xiv+257},
      ISBN = {0-486-42145-7},
   MRCLASS = {60-01 (60J10 60K25 60K35)},
  MRNUMBER = {1924231},
}

@book{b01book,
  title={Time series: data analysis and theory},
  author={Brillinger, David R},
  year={2001},
  publisher={SIAM}
}

@article {dn07,
    AUTHOR = {Doukhan, Paul and Neumann, Michael H.},
     TITLE = {Probability and moment inequalities for sums of weakly
              dependent random variables, with applications},
   JOURNAL = {Stochastic Process. Appl.},
  FJOURNAL = {Stochastic Processes and their Applications},
    VOLUME = {117},
      YEAR = {2007},
    NUMBER = {7},
     PAGES = {878--903},
      ISSN = {0304-4149,1879-209X}
}
\end{document}